\documentclass[twoside]{amsart}

\numberwithin{equation}{section}
\usepackage{amsmath,amssymb,amsfonts,amsthm,latexsym}
\usepackage{graphicx,color,enumerate}
\usepackage[all]{xy}

\usepackage{array}
\newcolumntype{P}[1]{>{\centering\arraybackslash}p{#1}}

\newtheorem{theorem}{Theorem}[section]
\newtheorem{lemma}[theorem]{Lemma}
\newtheorem{proposition}[theorem]{Proposition}
\newtheorem{corollary}[theorem]{Corollary}

\theoremstyle{definition}
\newtheorem{definition}[theorem]{Definition}

\theoremstyle{remark}
\newtheorem{remark}[theorem]{Remark}

\numberwithin{equation}{section}

\newcommand{\C}{\mathbb{C}}

\newcommand{\Q}{\mathbb{Q}}
\newcommand{\R}{\mathbb{R}}

\newcommand{\Z}{\mathbb{Z}}

\newcommand{\K}{\mathbb{K}}

\newcommand{\id}{\operatorname{id}}

\newcommand{\diag}{\operatorname{diag}}
\newcommand{\co}{\colon\thinspace}

\newcommand{\Hilb}{\operatorname{Hilb}}

\newcommand{\calF}{\mathcal{F}}

\begin{document}

\title[Laurent coefficients of a Gorenstein algebra]
{The Laurent coefficients of the Hilbert series of a Gorenstein algebra}

\author{Hans-Christian Herbig}
\address{Departamento de Matem\'{a}tica Aplicada,
Av. Athos da Silveira Ramos 149, Centro de Tecnologia - Bloco C, CEP: 21941-909 - Rio de Janeiro, Brazil}
\email{herbig@labma.ufrj.br}

\author{Daniel Herden}
\address{Department of Mathematics, Baylor University,
One Bear Place \#97328,
Waco, TX 76798-7328, USA}
\email{Daniel\_Herden@baylor.edu}

\author{Christopher Seaton}
\address{Department of Mathematics and Computer Science,
Rhodes College, 2000 N. Parkway, Memphis, TN 38112}
\email{seatonc@rhodes.edu}

\thanks{
CS was supported by the E.C. Ellett Professorship in
Mathematics.
}

\dedicatory{Dedicated to the memory of Leslie Tanner Hammontree.}
\keywords{Gorenstein algebras, Euler polynomials, Bernoulli numbers}
\subjclass[2010]{Primary 05A15; Secondary 11B68, 13H10, 13A50.}


\begin{abstract}
By a theorem of R. Stanley, a graded Cohen-Macaulay domain $A$ is \emph{Gorenstein} if and only if its Hilbert series satisfies the functional equation
\[\Hilb_A(t^{-1})=(-1)^d t^{-a}\Hilb_A(t),\]
where $d$ is the Krull dimension and $a$ is the a-invariant of $A$. We reformulate this
functional equation in terms of an infinite system of linear constraints on the Laurent coefficients of $\Hilb_A(t)$ at $t=1$. The main idea consists of examining the graded algebra $\mathcal F=\bigoplus_{r\in \Z}\mathcal F_r$ of formal power series in the variable $x$ that fulfill the condition $\varphi(x/(x-1))=(1-x)^r\varphi(x)$. As a byproduct, we derive quadratic and cubic relations for the Bernoulli numbers. The cubic relations have a natural interpretation in terms of coefficients of the Euler polynomials.
For the special case of degree $r=-(a+d)=0$, these results have been investigated previously by the authors and involved merely even Euler polynomials.
A link to the work of H. W. Gould and L. Carlitz  on power sums of symmetric number triangles is established.
\end{abstract}

\maketitle

\tableofcontents


\section{Introduction}
\label{sec:Intro}

Let $\K$ be a field. By a \emph{positively graded $\K$-algebra} we mean a graded $\K$-algebra
$A=\bigoplus_{i=0}^\infty A_i$ such that $A_0=\K$ and for all $i$ we have $\dim_\K(A_i)<\infty$.
To a positively graded $\K$-algebra $A=\bigoplus_{i=0}^\infty A_i$ one associates its \emph{Hilbert series} $\Hilb_A(t)\in\Z[\![t]\!]$, i.e., the generating function
\begin{align}\label{eq:HilbertSeries}
\Hilb_A(t):=\sum_{i=0}^\infty\dim_\K(A_i)\:t^i
\end{align}
counting the dimensions of the homogeneous components $A_i$.
If $A$ is finitely generated, then the Hilbert series is actually a rational function $\Hilb_A(t)=:P(t)/Q(t)$.
Moreover, the pole order $d$ in the Laurent expansion
\begin{align}\label{eq:LaurentExpansion}
\Hilb_A(t)=\sum_{i=0}^\infty\frac{\gamma_i}{(1-t)^{d-i}}
\end{align}
equals the Krull dimension of $A$. See \cite{BrunsHerzog}, \cite[Section 1.4]{DerskenKemperBook} or \cite[Section 3.10]{PopovVinberg} for more details.

By a theorem of R. Stanley \cite{BrunsHerzog,Stanley}, a graded Cohen-Macaulay domain $A$ is \emph{Gorenstein} if and only if its Hilbert series satisfies the
\emph{functional equation}
\begin{align}\label{eq:FunctionalEquation}
\Hilb_A(t^{-1})=(-1)^d t^{-a}\Hilb_A(t).
\end{align}
The number $a$ is the so-called \emph{a-invariant} of $A$ and can be understood as $\deg(P)-\deg(Q)$, see
\cite[Definition 5.1.5]{VillarrealMonomAlgs}. It is well-known that
\begin{align}\label{eq:degree}
r:=2\gamma_1/\gamma_0=-(a+d),
\end{align}
see \cite[Equation (3.32)]{PopovVinberg}.
We will call this quantity the \emph{degree}, as it will play this role for a $\mathbb Z$-graded algebra in this work. In the literature (e.g. in \cite{BensonBook}), $\gamma_0$ is frequently referred to as the degree, and we hope this will not lead to any confusion.
The degree $r$ has a natural interpretation in terms of the \emph{canonical module} $\omega_A$ of $A$ (see e.g. \cite{BensonBook, BrunsHerzog}) as
the degree shift in $\omega_A\cong A[r]$, where the shifted module $M[r]$ of a graded module $M=\bigoplus_{i}M_i$ is defined via $M[r]_i:=M_{r+i}$.
In this context, $A$ is Gorenstein if and only if the canonical module $\omega_A$ is one-dimensional.

The starting point and initial motivation of this work is a reformulation of the functional equation
\eqref{eq:FunctionalEquation} in terms of an infinite system of linear constraints  on the Laurent
coefficients $\gamma_i$. In a previous paper \cite{HHS}, we have treated the case of degree $r$ equal
zero, i.e. $a+d=0$. In the general case, the actual shape of the relations depends on the sign and parity
of the degree $r$.

\begin{theorem}
\label{thrm:Relations}
Let $H(t)=\sum_{i=0}^\infty \gamma_i (1-t)^{i-d}$ be a formal Laurent series around $t=1$ of pole order at most $d$. Then
$H(t)$ satisfies  the functional equation \eqref{eq:FunctionalEquation} if and only if the following conditions are
fulfilled, where we stipulate $\gamma_i = 0$ for $i < 0$.

If the degree $r$ is even, then for each $m \geq 1$:
\begin{align}\label{eq:DegreePositiveEven}
    \sum\limits_{i=0}^{m - 1} (-1)^i {m - 1 \choose i} \gamma_{m-r+i}
    = 0.
\end{align}
If the degree $r> 0$ is odd, then for each $m \geq 1$:
\begin{align}\label{eq:DegreePositiveOdd}
      \sum\limits_{i=0}^{m} (-1)^i {2m +r - 2 \choose m - i}{m+i \choose i} \gamma_{m + i - 1}
    = 0.
\end{align}

When $r<0$, then for each
$1 \leq m \leq  \lceil-r/2\rceil$:
\begin{align}\label{eq:DegreeAnyFirst}
    \sum\limits_{i=0}^{m} {m\choose i} {1-r\choose m+i}^{-1} \gamma_{m+i-1} = 0,
\end{align}
regardless of the parity of $r$. Moreover, if $r<0$ is odd, then for each $m\ge 1$:
\begin{align}
\label{eq:DegreeNegativeOdd}
    \sum\limits_{i=0}^m (-1)^i {m\choose i} \frac{m+i}m\:\gamma_{m-r+i} = 0.
\end{align}
\end{theorem}

Note that the relations among the Laurent coefficients are only unique up to scalar factors.
Alternate choices of the scaling will arise in the course of Section \ref{sec:ProofRelations}
and be considered as well in Section \ref{sec:CoefTriang}.

The relations given in Theorem \ref{thrm:Relations} in the case $r=0$ were first observed in
\cite[Section 8]{HerbigSeatonEM} in the context of experimental computations of the Laurent
coefficients of Hilbert series of a certain class of rings related to symplectic quotients by
the circle. An understanding of these relations suggested that these rings are all Gorenstein
of degree $r = 0$ (sometimes called \emph{graded Gorenstein} or \emph{strongly Gorenstein}),
which was then verified in \cite[Section 4]{HHS}. The initial purpose of Theorem \ref{thrm:Relations}
is to demonstrate that analogous relations occur for Gorenstein rings of arbitrary degree $r$.
Hence, one application of Theorem
\ref{thrm:Relations} is to contexts in which the $\gamma_i$ are more directly computable or
otherwise more accessible than the rational expression of the Hilbert series. The relations for
small values of $m$ can then be used to prove that a ring is not Gorenstein without computing
a more complete description of the ring or its Hilbert series.
We illustrate an example of such a context in Section~\ref{sec:Molien} for invariants of
finite groups.

In addition, the connection between the relations of Theorem \ref{thrm:Relations} and the functional
equation \eqref{eq:FunctionalEquation} has interesting consequences for some combinatorial
number sequences and other combinatorial constructions.
After introducing and describing the structure of the $\Z$-graded algebra $\calF=\bigoplus_{r\in\Z}\calF_r$
in Section \ref{sec:Relations}, the principal tool used in the proof of Theorem \ref{thrm:Relations}
that is then given in Section \ref{sec:ProofRelations}, we turn our attention to these consequences.
In Section \ref{sec:EulerPoly}, we reformulate the constraints of Theorem~\ref{thrm:Relations} to
express the odd coefficients $\gamma_{2i+1}$ in terms of the even coefficients $\gamma_{2i}$ and
vice versa, see Theorems \ref{thrm:OddFromEven} and \ref{thrm:EvenFromOdd}. These reformulations are
most succinctly stated using the coefficients of Euler polynomials as well as the Bernoulli numbers.
Hence, the algebra structure of $\calF$ yields a unified proof of a large collection of quadratic
and cubic identities for the Bernoulli numbers in Section \ref{sec:Identities}, see Theorems
\ref{thrm:Identities1} and \ref{thrm:Identities2}.

In Section \ref{sec:PowerSum}, we consider another application to combinatorics, connecting
the functional equation \eqref{eq:FunctionalEquation} to the power sum identities for symmetric
number arrays (e.g. Pascal's triangle)
developed by H. W. Gould \cite{Gould} and L. Carlitz \cite{Carlitz}. Specifically, by
demonstrating that a generating function for these power sums is an element of $\calF_1$,
we re-derive these power sum identities as corollaries of the case $r=1$ of Theorem \ref{thrm:Relations}
and its reformulations in Section \ref{sec:EulerPoly}. To complete the paper,
we illustrate the coefficient triangles that appear in Theorem \ref{thrm:Relations} by
displaying specific examples and consider the rescaling of the rows in Section \ref{sec:CoefTriang}.

Note that some of the auxiliary lemmas in this paper can be derived as special cases of known identities.
For instance, Lemma 4.5 can be seen to be a consequence of \cite[Equation (15.2.5)]{AbramowitzStegun}.
However, for the benefit of the reader, we provide elementary proofs.

Finally, let us observe that there exist Gorenstein algebras for each degree $r\in \Z$.
For example, a hypersurface of degree $k$ in affine space is Gorenstein of degree $r=-k$,
while invariant rings of unimodular finite group representations with no pseudoreflections
are Gorenstein of degree $r=0$ by \cite{WatanabeGor1,WatanabeGor2}, see Section \ref{sec:Molien}.
On the other hand, a polynomial ring $\K[x_1,x_2,\dots, x_d]$ in variables of degree
$\deg(x_i)\ge 1$ has $r=\sum_i (\deg(x_i)-1)$.


\section*{Acknowledgements}
HCH and CS would like to thank Baylor University, and CS would like to thank
the Universidade Federal do Rio de Janeiro, for their hospitality during work on
this project. We would also like to thank Emily Cowie, whose work on a separate
project \cite{CHHS} yielded a large family of examples used to discover some of
the results contained here.


\section{The $\gamma_i$ for invariants of finite groups}
\label{sec:Molien}

In this section, we give an example of an application of Theorem \ref{thrm:Relations}
to computational invariant theory. We refer the reader to \cite[Sections 2.5--6]{BensonBook},
\cite[Section 2.6]{DerskenKemperBook}, or \cite[Section 2.2]{SturmfelsBook} for
background on the topic considerer here. For simplicity, we work over $\C$.

Let $V$ be a vector space over $\C$ of dimension $n$, and let $G$ be a finite subgroup of
$\operatorname{GL} (V)$. By Molien's formula, the Hilbert series of the ring $\C[V]^G$
of $G$-invariant polynomials is given by
\begin{equation}
\label{eq:Molien}
    \Hilb_{\C[V]^G}(t) =
    \frac{1}{|G|} \sum\limits_{g\in G} \frac{1}{\det( \id - t g )}.
\end{equation}
It is a well-known consequence of this result that the first two coefficients of the Laurent
expansion of $\Hilb_{\C[V]^G}(t)$ at $t = 1$ are equal to
\[
    \gamma_0 = \frac{1}{|G|}
    \quad\quad\mbox{and}\quad\quad
    \gamma_1 = \frac{p}{2|G|},
\]
where $p$ is the number of \emph{pseudoreflections} in $G$, elements of $G$ whose fixed set in
$V$ has codimension $1$. By the same method used to determine these coefficients, we now demonstrate that
the $\gamma_k$ for $k \leq n$ can be computed using only those elements of $G$
that fix a subset of codimension at most $k$.

For $g \in G$, let $\mu_1(g),\ldots,\mu_n(g)$ denote the eigenvalues
of $g$, where we assume that any eigenvalues with value $1$ occur last on
this list. Choosing for each $g$ a basis for $V$ with respect to which
$g$ is diagonal, Equation \eqref{eq:Molien} becomes
\begin{equation}
\label{eq:MolienDiag}
    \Hilb_{\C[V]^G}(t) =
    \frac{1}{|G|} \sum\limits_{g\in G} \frac{1}{\prod_{j=1}^n \big(1 - \mu_j(g)t\big)}.
\end{equation}
Let $p(g)$ denote the number of $j$ such that $\mu_j(g) = 1$.
For $0 \leq k \leq n$, let $G_k = \{ g \in G : p(g) = k \}$,
$G_{\leq k} = \{ g \in G : p(g) \leq k \}$, and $G_{\geq k} = \{ g \in G : p(g) \geq k \}$.
As the term in Equation \eqref{eq:MolienDiag} corresponding to an element $g \in G$ has a pole
order equal to $p(g)$ at $t=1$, the coefficient $\gamma_k$ of the Laurent series depends only
on the elements of $G_{\geq n-k}$. Specifically, we can express
\begin{align*}
    \Hilb_{\C[V]^G}(t)
    &=
    \frac{1}{|G|}\sum\limits_{k=0}^n (1 - t)^{k-n}
    \sum\limits_{g\in G_{n-k}}
    \frac{1}{\prod_{j=1}^k \big( 1 - \mu_j(g)t \big)}
    \\&=
    \frac{1}{|G|}\sum\limits_{k=0}^n (1 - t)^{k-n}
    \sum\limits_{g\in G_{n-k}} \prod\limits_{j=1}^k
    \sum\limits_{i=0}^\infty \frac{- \mu_j(g)^i}{\big( \mu_j(g) - 1 \big)^{i+1}}(1 - t)^i.
\end{align*}
Then as $G_n = \{\id\}$, the sum over $g\in G_n$ is simply equal to $1$, yielding
$\gamma_0 = 1/|G|$. Similarly,
\begin{align*}
    \Hilb_{\C[V]^G}(t)
    &=
    \frac{1}{|G|}(1 - t)^{-n} + \frac{1}{|G|}(1 - t)^{1-n}
        \sum\limits_{g\in G_{n-1}} \frac{1}{1 - \mu_1(g)t}
    \\&\quad\quad
        + \frac{1}{|G|} \sum\limits_{g\in G_{\leq n-2}}
        \frac{1}{\prod_{j=1}^n \big( 1 - \mu_j(g)t\big)},
\end{align*}
where the last sum has a pole at $t=1$ of order at most $n-2$. In particular,
\begin{align*}
   \gamma_1 &=
    \frac{1}{|G|}\sum\limits_{g\in G_{n-1}} \frac{-1}{\mu_1(g) - 1}
    =\frac{1}{2|G|}\sum\limits_{g\in G_{n-1}} \left( \frac{-1}{\mu_1(g) - 1} + \frac{-1}{\mu_1(g^{-1}) - 1}\right)\\
    &=\frac{1}{2|G|}\sum\limits_{g\in G_{n-1}} \left( \frac{-1}{\mu_1(g) - 1} + \frac{-1}{\mu_1(g)^{-1} - 1}\right)
    =\frac{1}{2|G|}\sum\limits_{g\in G_{n-1}}1 =\frac{|G_{n-1}|}{2|G|}=\frac{p}{2|G|}.
\end{align*}
Continuing in this way,
\[
    \gamma_2    =
    \frac{1}{|G|}\left( \sum\limits_{g\in G_{n-1}} \frac{-\mu_1(g)}{\big( \mu_1(g) - 1\big)^2}
    + \sum\limits_{g\in G_{n-2}} \frac{1}{\big( \mu_1(g) - 1\big)\big( \mu_2(g) - 1\big)} \right),
\]
\begin{align*}
    \gamma_3
    &=
    \frac{1}{|G|}\left( \sum\limits_{g\in G_{n-1}} \frac{-\mu_1(g)^2}{\big( \mu_1(g) - 1\big)^3}
    + \sum\limits_{g\in G_{n-2}} \frac{\mu_1(g)+\mu_2(g)-2\mu_1(g)\mu_2(g)}
        {\big( \mu_1(g) - 1\big)^2\big( \mu_2(g) - 1\big)^2} \right.
    \\&\quad\quad\quad\quad
    + \left.\sum\limits_{g\in G_{n-3}} \frac{-1}{\big(\mu_1(g) - 1\big)\big(\mu_2(g) - 1\big)\big(\mu_3(g) - 1\big)}
    \right),
\end{align*}
etc.

To apply Theorem \ref{thrm:Relations}, given a finite group $G$, the value of $r$
is determined from $\gamma_0$ and $\gamma_1$ using Equation \eqref{eq:degree}.
Note that in this context, $r = p = |G_{n-1}|\ge 0$ is the number of pseudoreflections in $G$.
Furthermore, Equation \eqref{eq:degree} is reflected by Theorem \ref{thrm:Relations}
as Equation \eqref{eq:DegreePositiveEven} with $m= \frac r2 +1$  for $r$ even and
as Equation \eqref{eq:DegreePositiveOdd} with $m=1$ for $r$ odd. Then, staying for the moment with the
case $r$ even, the constraint in Theorem \ref{thrm:Relations}
corresponding to $m=\frac r2 + 2$ gives a necessary condition for $\C[V]^G$ to be Gorenstein that involves
only the eigenvalues of elements of $G_{\geq n-3}$. Similarly, the constraint corresponding
to arbitrary $\frac r2 + m$ gives a necessary Gorenstein condition for $\C[V]^G$ involving only
$G_{\geq n-2m+1}$. Note that if $G$ contains no pseudoreflections, then $\C[V]^G$ is Gorenstein
if and only if $G\leq \operatorname{SL}(V)$ by a Theorem of Watanabe \cite[Theorem 1]{WatanabeGor2};
hence, the Gorenstein property of $\C[V]^G$ can be established much more easily in this case.

As an explicit example, let $\zeta$ be a primitive $6$th root of unity and consider the subgroup $G$
of $\operatorname{GL}_4 (\C)$ of order $12$ generated by
$a = \diag(\zeta,\zeta^2,\zeta,1)$ and $b = \diag(1,1,1,-1)$. Clearly,
$\gamma_0 = 1/12$, and as $G$ contains the single pseudoreflection
$b$, we have $\gamma_1 = 1/24$. It follows
that if $\C[\C^4]^G$ were to be Gorenstein, we must have $r = 1$ so that $\gamma_1 - 3\gamma_2 + 2\gamma_3 = 0$
by Equation \eqref{eq:DegreePositiveOdd} with $m = 2$. However, as described above,
one may easily compute using only the elements of $G_4 = \{\id\}$, $G_3 = \{ b \}$, $G_2 = \{ a^3 \}$,
and $G_1 = \{ a, a^2, a^4, a^5, a^3 b\}$ that $\gamma_2 = 1/24$ and $\gamma_3 = 1/72$ so that
$\gamma_1 - 3\gamma_2 + 2\gamma_3 = -1/18$ and hence $\C[\C^4]^G$
is not Gorenstein. In this simple example, we can conclude that the Gorenstein property fails with a
computation involving only the $8$ elements of $G_{\geq 1}$, and in particular without computing the invariant ring or
its Hilbert series completely. In larger, less contrived examples, $G_{\geq 1}$ may be much smaller
relative to the size of $G$.


\section{Construction of the algebra $\calF$}
\label{sec:Relations}


For the remainder of this paper, let $\K$ denote one of the fields $\Q$, $\R$, or $\C$.

Let us assume that $H(t)$ is a formal Laurent series over the field $\K$ around $t=1$ of pole order at most $d$.
Clearly, the substitution $\varphi(x):=x^d H(1-x)$ is a formal power series in the variable $x$, i.e. an element
of $\K[\![x]\!]$. Conversely, a formal power series $\varphi(x)$ defines a formal Laurent series
$H(t):=(1-t)^{-d}\varphi(1-t)$. Assuming that $H(t)$ satisfies the functional equation \eqref{eq:FunctionalEquation},
we derive
\begin{align*}
    \frac{(-1)^dt^{-a}}{(1-t)^d}\varphi(1-t)
    =   H(t^{-1})
    =   \frac{1}{(1-t^{-1})^d}\varphi(1-t^{-1})
    =   \frac{(-t)^d}{(1-t)^d}\varphi((t-1)/t).
\end{align*}
Multiplying by $t^{-d}(t-1)^d$ and substituting $x=1-t$, we find that $\varphi(x)$ satisfies
the functional equation
\begin{align}
\label{eq: power series functional equation}
\varphi(x/(x-1))=(1-x)^r\varphi(x),
\end{align}
where $r=-(a+d)$.
\begin{definition} Let $\calF_r$ be the space of formal power series in the variable $x$ satisfying
Equation \eqref{eq: power series functional equation}. We introduce the $\mathbb Z$-graded vector space
$\calF=\bigoplus_{r\in \Z}\calF_r\subset \K[\![x]\!]$.
\end{definition}

Using the above computations, it is an easy task to verify the following statements, which we leave to the reader.

\begin{proposition} \label{prop:GradedAlgebra}
The substitution $H(t)\mapsto \varphi(x):=x^d H(1-x)$ establishes an isomorphism between the
space of formal Laurent series of pole order at most $d$ in the variable $t$ around $t=1$ satisfying Equation
\eqref{eq:FunctionalEquation} and $\calF_r$. With respect to the Cauchy product of formal power
series, $\calF=\bigoplus_{r\in \Z}\calF_r$ forms a $\Z$-graded algebra.
\end{proposition}

In \cite{HHS}, the authors have investigated the algebra $\calF_0$, i.e., the algebra of formal power series
in the variable $x$ invariant under the M\"obius transformation $x\mapsto x/(x-1)$. We recall the following.
\begin{theorem}[\cite{HHS}]
\label{thrm:Deg0HHS}
For a formal power series $\varphi(x)=\sum_{i=0}^\infty \gamma_i\,x^i$, the following conditions are equivalent:
\begin{enumerate}[(i.)]
\item $\varphi(x)\in \calF_0$.
\item $\varphi(x)$ is a formal composite with $\lambda_1(x):=x^2/(1-x)$, i.e., there exists a $\rho(y)\in\K[\![y]\!]$ such that $\varphi(x)=\rho(x^2/(1-x))$.
\item For each $m\ge 1$, the relation $\sum_{i=0}^{m - 1} (-1)^i {m - 1 \choose i} \gamma_{m + i}
    = 0$ holds.
\end{enumerate}
\end{theorem}

In fact, in (ii.), instead of $\lambda_1(x):=x^2/(1-x)$, we could have worked with any
$\lambda(x) \in \calF_0\cap \mathfrak m^2$ whose class in $\mathfrak m^2/\mathfrak m^3$
is nonzero. Here $\mathfrak m=x\K[\![x]\!]$ is the maximal ideal of $\K[\![x]\!]$.
This was noted in \cite[Remark 2.4]{HHS}, where we constructed examples of such series that
are related to the Genocchi sequence. Another natural example slipped through our attention,
namely:
\begin{align}
\label{eq:xox-2sqr}
    \lambda_2(x):=\left(\frac{x}{x-2}\right)^2\in \calF_0\cap \mathfrak m^2.
\end{align}

Moreover, the reader can readily check that
\begin{align}
g_{-1}(x):=x-2\in\calF_{-1}
\end{align}
is invertible.  Hence, for any $\varphi(x)\in \calF_r$, it follows that $\varphi(x)(x-2)^r\in \calF_0$. Using $\lambda_2(x)$, we can write $\varphi(x)$ uniquely in the form
\begin{align}
\label{eq:membF_r}
\varphi(x)=\sum_{i= 0}^\infty \delta_i\, x^{2i}\,(x-2)^{-r-2i}
\end{align}
for $\delta_i \in \K$. Conversely, every such series is an
element of $\calF_r$.

\begin{theorem} \label{thrm:LaurentPoly} The $\Z$-graded algebra $\calF=\bigoplus_{r\in\Z}\calF_r$ is isomorphic
to the algebra of Laurent polynomials $\calF_0(g_{-1})$ with coefficients in $\calF_0$ in the variable $g_{-1}$ of degree $-1$.
\end{theorem}


\section{Proof of Theorem \ref{thrm:Relations}}
\label{sec:ProofRelations}

Let $r\in\Z$, and let $\varphi(x) = \sum_{i= 0}^\infty \gamma_i x^i$.
In this section, we prove Theorem \ref{thrm:Relations}, that $\varphi(x) \in \mathcal{F}_r$
is equivalent to the relations described in Equations \eqref{eq:DegreePositiveEven},
\eqref{eq:DegreePositiveOdd}, \eqref{eq:DegreeAnyFirst}, and \eqref{eq:DegreeNegativeOdd}.
In the case $r < 0$, Equations \eqref{eq:DegreePositiveEven} and \eqref{eq:DegreeNegativeOdd}
involve only the coefficients $\gamma_i$ for $i \geq 1-r$, while Equation \eqref{eq:DegreeAnyFirst}
involves the $\gamma_i$ for $i \leq -r-1$ when $r$ is even and $i \leq -r$ when $r$ is odd
(in the case of $r$ even, there is no relation involving $\gamma_{-r}$).
Hence, we will first assume in this case that $\gamma_i = 0$ for $i \leq -r$ in
Subsection \ref{subsec:ProofInfiniteRelat} and deal with Equation \eqref{eq:DegreeAnyFirst}
separately in Subsection \ref{subsec:ProofFirstRelat}.

\subsection{Equations \eqref{eq:DegreePositiveEven}, \eqref{eq:DegreePositiveOdd}, and
\eqref{eq:DegreeNegativeOdd}}
\label{subsec:ProofInfiniteRelat}

We first consider the case of $r$ even with the assumption that the first $1-r$ coefficients vanish.

\begin{lemma}
\label{lem:EvenJustShifted}
Let $\varphi(x) = \sum_{i= 0}^\infty \gamma_i x^i \in \K[\![x]\!]$ and let $r$ be an
even integer. If $r < 0$, assume that $\gamma_i = 0$ for each $i \leq -r$.
Then $\varphi(x)\in\mathcal{F}_r$ if and only if
$\overline{\varphi}(x) := x^r\varphi(x) \in \mathcal{F}_0$.
\end{lemma}
\begin{proof}
Note that when $r < 0$, the assumption that the first $1-r$ of the $\gamma_i$ vanish ensures
that $\overline{\varphi}(x)$ is a power series. By Equation \eqref{eq: power series functional equation},
$\overline{\varphi}(x) \in \mathcal{F}_0$ if and only if $\overline{\varphi}(x/(x-1)) = \overline{\varphi}(x)$,
which by a simple computation using $\overline{\varphi}(x) = x^r\varphi(x)$ and the fact that $r$ is even is
equivalent to $\varphi(x/(x-1)) = (1 - x)^r \varphi(x)$, i.e. $\varphi(x)\in\mathcal{F}_r$.
\end{proof}

\begin{corollary}
\label{cor:RelPosEven}
Let $\varphi(x) = \sum_{i= 0}^\infty \gamma_i x^i \in \K[\![x]\!]$ and let $r$ be an
even integer. If $r < 0$, assume that $\gamma_i = 0$ for each $i \leq -r$.
Then the $\gamma_i$ satisfy Equation \eqref{eq:DegreePositiveEven} for
each $m \geq 1$ if and only if $\varphi(x) \in \mathcal{F}_r$.
\end{corollary}
\begin{proof}
If $r = 0$, then this claim corresponds to the equivalence of (i) and (iii) in Theorem \ref{thrm:Deg0HHS}.
If $r\neq 0$, set $\overline{\varphi}(x) = x^r\varphi(x) = \sum_{i=0}^\infty \overline{\gamma}_i x^i$.
Equation \eqref{eq:DegreePositiveEven} for all $m \geq 1$ is equivalent to
\[
    \sum\limits_{i=0}^{m-1} (-1)^i {m - 1 \choose i} \overline{\gamma}_{m+i} = 0,
\]
which is equivalent to $\overline{\varphi}(x) \in \mathcal{F}_0$ by the above argument.
By Lemma \ref{lem:EvenJustShifted}, this is equivalent to $\varphi(x)\in\mathcal{F}_r$.
\end{proof}

To proceed to the odd case, we start with the following useful characterization of solutions for
the constraints in Theorem \ref{thrm:Relations} when $r$ is positive and odd, i.e. $r = 2k+1$ for some $k \geq 0$.
Note that Equation \eqref{eq:RelPosOddK} simply rewrites Equation \eqref{eq:DegreePositiveOdd}
replacing $r = 2k+1$.

\begin{lemma}
\label{lem:RelPosOddCharacterize}
Let $\varphi(x) = \sum_{i= 0}^\infty \gamma_i x^i \in \K[\![x]\!]$, and let $k\geq 0$
be an integer. Then
\begin{equation}
\label{eq:RelPosOddK}
    \sum\limits_{i=0}^m(-1)^i {2m + 2k - 1 \choose m - i} {m + i \choose i}
        \gamma_{m+i-1}  =   0
\end{equation}
for each $m \geq 1$ if and only if there is a power series $\psi(y) \in \K[\![y]\!]$
such that
\begin{equation}
\label{eq:RelPosOddChar}
    \sum\limits_{i= 0}^\infty \gamma_i x^i = \frac{1}{x} \cdot \frac{d^{2k-1}}{dx^{2k-1}}
        \left[\frac{x^{2k}}{1-x} \psi\left( \frac{x^2}{1 - x} \right) \right],
\end{equation}
where, for $k = 0$, we interpret $\frac{d^{-1}}{dx^{-1}}$ as the indefinite integral (with
vanishing constant of integration).
If $k \geq 1$, then Equation \eqref{eq:RelPosOddChar} is equivalent to
\begin{equation}
\label{eq:RelPosOddCharKG1}
    \sum\limits_{i= 0}^\infty \gamma_i x^i = \frac{1}{x} \cdot \frac{d^{2k-1}}{dx^{2k-1}}
        \left[x^{2k-2} \rho\left( \frac{x^2}{1 - x} \right) \right]
\end{equation}
for some $\rho(y)\in \K[\![y]\!]$.

\end{lemma}
\begin{proof}
For each $m$, multiplying both sides by the scalar $\frac{(m!)^2}{(2m+2k-1)!}$, we
rewrite Equation \eqref{eq:RelPosOddK} as
\begin{equation}
\label{eq:RelPosOddKrewrite}
    \sum\limits_{i=0}^m (-1)^i {m\choose i} \frac{(m+i)!}{(m+2k+i-1)!} \gamma_{m+i-1} = 0.
\end{equation}
Define $\widetilde{\varphi}(x) = \sum_{i=0}^\infty \widetilde{\gamma}_i x^i$ by
$\widetilde{\gamma}_0 = \widetilde{\gamma}_1 = 0$ and
$\widetilde{\gamma}_{i+2} = \frac{(i+1)!}{(i+2k)!} \gamma_i$, and then for each $m$,
the $\gamma_i$ satisfy Equation \eqref{eq:RelPosOddKrewrite} if and only if
\[
    \sum\limits_{i=0}^m (-1)^i {m \choose i} \widetilde{\gamma}_{m+i+1} = 0,
\]
which, along with $\widetilde{\gamma}_1 = 0$ is equivalent to
$\widetilde{\varphi}(x) \in \mathcal{F}_0$. By Theorem \ref{thrm:Deg0HHS},
and as $\widetilde{\varphi}(x)$ has no constant nor linear terms by definition,
this is equivalent to the existence of a $\psi(y) \in \K[\![y]\!]$ such that
\[
    \widetilde{\varphi}(x)
        =   \sum_{i=0}^\infty \frac{(i+1)!}{(i+2k)!} \gamma_i x^{i+2}
        =   \left(\frac{x^2}{1-x}\right) \psi \left(\frac{x^2}{1-x}\right).
\]
Multiplying by $x^{2k-2}$ yields
\[
    \sum_{i=0}^\infty \frac{(i+1)!}{(i+2k)!} \gamma_i x^{i+2k}
        =   \left(\frac{x^{2k}}{1-x}\right) \psi \left(\frac{x^2}{1-x}\right),
\]
i.e.
\[
    \sum_{i=0}^\infty \gamma_i x^{i+1}
        =   \frac{d^{2k-1}}{dx^{2k-1}}\left[ \left(\frac{x^{2k}}{1-x}\right)
                \psi \left(\frac{x^2}{1-x}\right) \right].
\]
Dividing both sides by $x$ yields Equation \eqref{eq:RelPosOddChar}.

If $k \geq 1$, then we can rewrite Equation \eqref{eq:RelPosOddChar} as
\[
    \sum\limits_{i= 0}^\infty \gamma_i x^i = \frac{1}{x} \cdot \frac{d^{2k-1}}{dx^{2k-1}}
        \left[x^{2k-2}\left(\frac{x^2}{1-x}\right)\psi\left( \frac{x^2}{1 - x} \right) \right],
\]
and Equation \eqref{eq:RelPosOddCharKG1} holds with $\rho(y) = y\psi(y)$. Note that the
derivative maps the constant term $\rho(0)$ of $\rho(y)$ to $0$, which makes the last step an equivalence.
\end{proof}

We next have the following, which demonstrates Theorem \ref{thrm:Relations} in the
case $r = 1$.

\begin{lemma}
\label{lem:Relr1}
A power series $\varphi(x) = \sum_{i= 0}^\infty \gamma_i x^i$ is an element of $\mathcal{F}_1$
if and only if the $\gamma_i$ satisfy
\begin{equation}
\label{eq:Relr1}
    \sum\limits_{i=0}^{m} (-1)^i {2m - 1 \choose m - i}{m + i \choose i} \gamma_{m + i - 1}
        = 0
\end{equation}
for each $m \geq 1$.
\end{lemma}
\begin{proof}
As above, by Equation \eqref{eq: power series functional equation},
$\rho(x) \in \mathcal{F}_0$ if and only if
\begin{equation}
\label{eq:r1}
    \rho\left(\frac{x}{x-1}\right) = \rho(x).
\end{equation}
Differentiating yields
\begin{equation}
\label{eq:r1Diff}
    \rho^\prime\left(\frac{x}{x-1}\right) = -(1 - x)^2 \rho^\prime(x).
\end{equation}

Let $\varphi(x)\in\mathcal{F}_1$, and then there is a $\rho(x) \in\mathcal{F}_0$
such that $\varphi(x) = \rho(x)/(x-2)$ by Theorem \ref{thrm:LaurentPoly}.
Define
\begin{equation}
\label{eq:DefChi}
    \chi(x) =   (1 - x) \frac{d}{dx}\left[x\varphi(x)\right],
\end{equation}
and then we may express
\[
    \varphi(x)  =   \frac{1}{x}\int\limits_0^x \frac{1}{1-\xi}\chi(\xi)\, d\xi ,
\]
where the integral is the formal integral of power series, interpreted term by term.
Substituting $\varphi(x) = \rho(x)/(x-2)$ into Equation \eqref{eq:DefChi} and applying the product rule,
we have
\[
    \chi(x) =   (1 - x)\left[\frac{x}{x-2} \rho^\prime(x) - \frac{2}{(x - 2)^2}\rho(x)\right],
\]
from which a simple computation using Equations \eqref{eq:r1} and \eqref{eq:r1Diff}
demonstrates that $\chi(x/(x - 1)) = \chi(x)$ so that $\chi(x) \in \mathcal{F}_0$.
Then by Theorem \ref{thrm:Deg0HHS}, it follows
that there is a formal power series $\psi(y) \in \K[\![y]\!]$ such that
$\chi(x) = \psi(x^2/(1 - x))$ and hence
\[
    \varphi(x)  =   \frac{1}{x}\int\limits_0^x \frac{1}{1-\xi}\psi\left(\frac{\xi^2}{1-\xi}\right)\, d\xi .
\]
This is precisely Equation \eqref{eq:RelPosOddChar} for the case $k=0$, so by
Lemma \ref{lem:RelPosOddCharacterize}, the $\gamma_i$ satisfy Equation
\eqref{eq:Relr1}.

Conversely, suppose the $\gamma_i$ satisfy Equation \eqref{eq:Relr1} so that by
Lemma \ref{lem:RelPosOddCharacterize} and Theorem \ref{thrm:Deg0HHS}, there is a
$\chi(x)\in\mathcal{F}_0$ such that
\[
    \varphi(x)  =   \frac{1}{x}\int\limits_0^x \frac{1}{1-\xi}\chi(\xi)\, d\xi .
\]
Let
\[
    \rho(x) =   (x-2)\varphi(x) =   \frac{x-2}{x}\int\limits_0^x \frac{1}{1-\xi}\chi(\xi)\, d\xi .
\]
By Theorem \ref{thrm:Deg0HHS}, $\chi(x) = \chi(x/(x-1))$ so that using the substitution $\zeta = \xi/(1-\xi)$ we have
\begin{align*}
    \rho\left(\frac{x}{x-1}\right) &=   -\frac{x-2}{x}\int\limits_0^{\frac{x}{x-1}} \frac{1}{1-\xi}\chi(\xi)\, d\xi
  = -\frac{x-2}{x}\int\limits_0^{\frac{x}{x-1}} \frac{1}{1-\xi}\chi\left(\frac{\xi}{\xi-1}\right) d\xi\\
& = \frac{x-2}{x}\int\limits_0^{x} \frac{1}{1-\zeta}\chi(\zeta)\, d\zeta = \rho(x).
\end{align*}
Hence $\rho(x) \in \mathcal{F}_0$, and hence $\varphi(x) \in\mathcal{F}_1$ by
Theorem \ref{thrm:LaurentPoly}.
\end{proof}

To proceed, we will need to establish the following identity.
We will use $k^{(r)} = k(k+1)(k+2)\cdots (k+r-1)$ to denote the rising Pochhammer symbol
and $(k)_r = k(k-1)(k-2)\cdots(k-r+1)$ to denote the falling Pochhammer symbol.

\begin{lemma}
\label{lem:DiffFormula}
Let $r$ be a positive integer and let $k, a \in \K$. Then
\begin{equation}
\label{eq:DiffFormula}
    \frac{d^{r+1}}{dx^{r+1}} \left[ x^r \left(\frac{x}{x-a}\right)^k\right]
    =
    (-a)^{r+1} k^{(r+1)}\frac{x^{k-1}}{(x-a)^{k+r+1}}.
\end{equation}
\end{lemma}
\begin{proof}
We first use the binomial series to rewrite
\begin{align*}
    x^r \left(\frac{x}{x-a}\right)^k
        &=      x^r \left(1 - \frac{a}{x}\right)^{-k}
        \\&=    x^r \sum\limits_{i=0}^\infty {-k \choose i} \left(- \frac{a}{x} \right)^i
        \\&=    \sum\limits_{i=0}^\infty {-k \choose i} (-a)^i x^{r-i}.
\end{align*}
Differentiating, we express the left-hand side of Equation \eqref{eq:DiffFormula} as
\begin{align*}
    \frac{d^{r+1}}{dx^{r+1}} \left[ x^r \left(\frac{x}{x-a}\right)^k\right]
        &=      \frac{d^{r+1}}{dx^{r+1}} \left[
                    \sum\limits_{i=0}^\infty {-k \choose i} (-a)^i x^{r-i} \right]\\
        &=      \frac{d^{r+1}}{dx^{r+1}} \left[
                    \sum\limits_{i=r+1}^\infty {-k \choose i} (-a)^i x^{r-i} \right]
        \\&=    \frac{d^{r+1}}{dx^{r+1}} \left[
                    \sum\limits_{i=0}^\infty {-k \choose i+r+1} (-a)^{i+r+1} x^{-i-1} \right]
        \\&=    \sum\limits_{i=0}^\infty {-k \choose i+r+1 }
                    (-a)^{i+r+1} (-1)^{r+1} (i+1)^{(r+1)} x^{-i-r-2}
        \\&=    \sum\limits_{i=0}^\infty \frac{(-k)_{r+i+1}}{(i+r+1)!}
                    (-a)^{i+r+1} (-1)^{r+1} \frac{(i+r+1)!}{i!} x^{-i-r-2}
        \\&=    k^{(r+1)} \sum\limits_{i=0}^\infty
                    \frac{(-k-r-1)_{i}}{i!} (-a)^{i+r+1} x^{-i-r-2}
        \\&=    (-a)^{r+1} k^{(r+1)} x^{-r-2}
                    \sum\limits_{i=0}^\infty {-k-r-1\choose i} (-a)^i x^{-i}
        \\&=    (-a)^{r+1} k^{(r+1)} x^{-r-2}
                    \left(\frac{x - a}{x}\right)^{-k-r-1}
        \\&=    (-a)^{r+1} k^{(r+1)}
                    \frac{x^{k-1}}{(x - a)^{k+r+1}}.
        \qedhere
\end{align*}
\end{proof}

With this, we now prove the following, which demonstrates Theorem \ref{thrm:Relations}
when $r > 1$ is odd, i.e. $r = 2k+1$ for $k \geq 1$.

\begin{lemma}
\label{lem:RelPosOddg1}
Let $\varphi(x) = \sum_{i= 0}^\infty \gamma_i x^i \in \K[\![x]\!]$, and let $k\geq 1$
be an integer. Then the $\gamma_i$ satisfy Equation \eqref{eq:RelPosOddK}, equivalently
\eqref{eq:DegreePositiveOdd} with $r = 2k+1$, if and only if $\varphi(x) \in \mathcal{F}_{2k+1}$.
\end{lemma}
\begin{proof}
As $k \geq 1$, we have by Lemma \ref{lem:RelPosOddCharacterize} that the $\gamma_i$ satisfy
Equation \eqref{eq:RelPosOddK} if and only if there is a power series $\rho(y)\in\K[\![y]\!]$
satisfying Equation \eqref{eq:RelPosOddCharKG1}. Using Theorem~\ref{thrm:Deg0HHS} and the
generator given in Equation \eqref{eq:xox-2sqr},
given $\rho(y)$, there is a power series $\chi(y) = \sum_{i=0}^\infty \delta_i y^i \in\K[\![y]\!]$
such that $\chi((x/(x-2))^2) = \rho(x^2/(1-x))\in \mathcal{F}_0$. That is, the $\gamma_i$ satisfy
Equation \eqref{eq:RelPosOddK} if and only if
\begin{align*}
    \sum\limits_{i= 0}^\infty \gamma_i x^i
        &=      \frac{1}{x} \cdot \frac{d^{2k-1}}{dx^{2k-1}}
                    \left[x^{2k-2} \chi\left( \left(\frac{x}{x - 2}\right)^2 \right) \right]
        \\&=    \frac{1}{x} \cdot \frac{d^{2k-1}}{dx^{2k-1}}
                    \left[x^{2k-2} \sum\limits_{i=0}^\infty \delta_i \left(\frac{x}{x - 2}\right)^{2i} \right]
        \\&=    \frac{1}{x} \sum\limits_{i=0}^\infty \delta_i \cdot \frac{d^{2k-1}}{dx^{2k-1}}
                    \left[x^{2k-2} \left(\frac{x}{x - 2}\right)^{2i} \right].
\end{align*}
Applying Lemma \ref{lem:DiffFormula} to each term and noting that the term $i=0$ vanishes,
we continue
\begin{align*}
    &=      \frac{1}{x} \sum\limits_{i=1}^\infty \delta_i (-2)^{2k-1}
                (2i)^{(2k-1)} \frac{x^{2i-1}}{(x-2)^{2i+2k-1}}
    \\&=    \frac{1}{(x-2)^{2k+1}} \sum\limits_{i=1}^\infty \delta_i (-2)^{2k-1}
                \frac{(2i + 2k - 2)!}{(2i - 1)!} \left(\frac{x}{x-2}\right)^{2i-2}
    \\&=    \frac{1}{(x-2)^{2k+1}} \sum\limits_{i=0}^\infty \delta_{i+1} (-2)^{2k-1}
                \frac{(2i + 2k)!}{(2i + 1)!} \left(\frac{x}{x-2}\right)^{2i}.
\end{align*}
By Theorem \ref{thrm:Deg0HHS} along with the generator of Equation \eqref{eq:xox-2sqr},
the factor that is a power series in $(x/(x-2))^2$ is an element of $\mathcal{F}_0$
so that by Theorem \ref{thrm:LaurentPoly},
this expression is an element of $\mathcal{F}_{2k+1}$, completing the proof.
\end{proof}

We next consider the case of $r < 0$ odd with the assumption that the first $1-r$
coefficients vanish.

\begin{lemma}
\label{lem:OddJustShifted}
Let $\varphi(x) = \sum_{i= 0}^\infty \gamma_i x^i \in \K[\![x]\!]$ and let $r$ be an odd
integer. If $r < 0$, assume that $\gamma_i = 0$ for each $i \leq -r$. Then $\varphi\in\mathcal{F}_r$
if and only if
$\overline{\varphi}(x) := x^{r-1}\varphi(x)\in\mathcal{F}_1$.
\end{lemma}
\begin{proof}
As in the proof of Lemma \ref{lem:EvenJustShifted}, $\overline{\varphi}(x)$
is a power series in each case due to the assumption that the first $\gamma_i$ vanish.
A simple computation using the fact that $r$ is odd demonstrates that
$\varphi(x/(x-1)) = (1 - x)^r \varphi(x)$ if and only if
$\overline{\varphi}(x/(x-1)) = (1 - x) \overline{\varphi}(x)$.
\end{proof}

\begin{corollary}
\label{cor:RelNegOdd}
Let $\varphi(x) = \sum_{i= 0}^\infty \gamma_i x^i \in \K[\![x]\!]$ and let $r$ be an odd
negative integer. Assume that $\gamma_i = 0$ for each $i \leq -r$. Then $\varphi\in\mathcal{F}_r$
if and only if the $\gamma_i$ satisfy Equation \eqref{eq:DegreeNegativeOdd}.
\end{corollary}
\begin{proof}
Define $\overline{\varphi}(x) = x^{r-1}\varphi(x)  = \sum_{i=0}^\infty \overline{\gamma}_i x^i$.
Specializing Equation \eqref{eq:DegreePositiveOdd} for the $\overline{\gamma}_i$ to the case $r = 1$,
we have
\[
    \sum\limits_{i=0}^{m} (-1)^i {2m - 1 \choose m - i}{m+i \choose i} \overline{\gamma}_{m + i - 1}
    = 0, \quad\quad m\geq 1.
\]
Using $\overline{\gamma}_i = \gamma_{i-r+1}$ and multiplying both sides of the equation
by the constant $m!(m-1)!/(2m-1)!$, this is equivalent to Equation \eqref{eq:DegreeNegativeOdd},
which along with Lemma~\ref{lem:OddJustShifted} completes the proof.
\end{proof}

With this, Theorem \ref{thrm:Relations} follows when, in the case $r < 0$, it is assumed that
$\gamma_i = 0$ for $i \leq -r$. However, note that each $\mathcal{F}_r$ is obviously closed under addition,
and that a series $\varphi(x) = \sum_{i= 0}^\infty \gamma_i x^i$ is an element of $\mathcal{F}_r$ iff it is of the form
$\varphi(x)=\sum_{i= 0}^\infty \delta_i\, x^{2i}\,(x-2)^{-r-2i}$ for $\delta_i \in \K$,
cf. Equation \eqref{eq:membF_r}. In particular, $\varphi(x) = \varphi_1(x)+\varphi_2(x)$ decomposes naturally
with
\[
\varphi_1(x)=\sum_{i= 0}^{-r} \gamma_i x^i =\sum_{i= 0}^{\lfloor -r/2\rfloor} \delta_i\, x^{2i}\,(x-2)^{-r-2i} \in \mathcal{F}_r
\]
a polynomial of degree at most $-r$ and
\[
\varphi_2(x)=\sum_{i= 1-r}^\infty \gamma_i x^i =\sum_{i= \lfloor -r/2\rfloor +1}^\infty \delta_i\, x^{2i}\,(x-2)^{-r-2i} \in \mathcal{F}_r.
\]
The same observation holds for power series satisfying the relations given in Theorem \ref{thrm:Relations}. Thus,
for the remainder of the proof, it is sufficient to demonstrate that, when $r < 0$, a polynomial of degree at most $-r$ is an element
of $\mathcal{F}_r$ if and only if its coefficients satisfy Equation \eqref{eq:DegreeAnyFirst}
for $1 \leq m \leq \lceil -r/2\rceil$.

\subsection{Equation \eqref{eq:DegreeAnyFirst}}
\label{subsec:ProofFirstRelat}

We now fix $r < 0$ and let $f(x) = \sum_{i=0}^{-r} \gamma_i x^i$ be a polynomial of degree at most $-r$.
The goal of this section is to demonstrate that $f(x) \in \mathcal{F}_r$ if and only if
the $\gamma_i$ satisfy Equation \eqref{eq:DegreeAnyFirst} for $1 \leq m \leq \lceil -r/2\rceil$.
Using the description of $\mathcal{F}_r$ given by Theorems \ref{thrm:Deg0HHS} and \ref{thrm:LaurentPoly}
and the generator given in Equation \eqref{eq:xox-2sqr},
we reformulate this statement into the following.

\begin{proposition}
\label{prop:NegFirstReformulate}
Assume $r < 0$ and let $f(x) = \sum_{i=0}^{-r}\gamma_i x^i$. Then the $\gamma_i$ satisfy
Equation \eqref{eq:DegreeAnyFirst} for $1 \leq m \leq\lceil -r/2\rceil$ if and only if
we may express
\begin{equation}
\label{eq:NegFirstReformulate}
    f(x)    =   (x - 2)^{-r} \sum\limits_{i=0}^{\lfloor -r/2 \rfloor} \delta_i
                                \left(\frac{x}{x-2}\right)^{2i}
            =   x^{-r} \sum\limits_{i=0}^{\lfloor -r/2 \rfloor} \delta_i
                                \left(\frac{x-2}{x}\right)^{-r-2i}
\end{equation}
for $\delta_i \in \K$.
\end{proposition}

To prove Proposition \ref{prop:NegFirstReformulate}, we will first need an auxiliary result.

\begin{lemma}
\label{lem:NegFirstBinomIdentity}
For $m,n\geq 0$, we have
\begin{equation}
\label{eq:NegFirstBinomIdentity}
    \sum\limits_{i=2n-m+1}^m {m \choose i} \frac{(m + i)!}{(m+i-2n-1)!}\:(-2)^{-i} = 0.
\end{equation}
\end{lemma}
\begin{proof}
Differentiating
\[
    (x^2+x)^{m}   = (x+1)^{m} x^{m} = \sum\limits_{i=0}^m {m\choose i} x^{m+i}
\]
yields
\begin{align*}
    \frac{d^{2n+1}}{dx^{2n+1}} \left[ (x^2+x)^{m} \right] &= \sum\limits_{i=0}^m {m\choose i} (m+i)_{2n+1} x^{m-2n+i-1}\\
    &= \sum\limits_{i=2n-m+1}^m {m \choose i} \frac{(m + i)!}{(m+i-2n-1)!}\:x^{m-2n+i-1}
\end{align*}
and
\begin{equation}
\label{eq:NegFirstBinomIdentity2}
    \sum\limits_{i=2n-m+1}^m {m \choose i} \frac{(m + i)!}{(m+i-2n-1)!}\:x^{i}=x^{2n-m+1}\frac{d^{2n+1}}{dx^{2n+1}} \left[ (x^2+x)^{m} \right].
\end{equation}
Note that
\[
    (x^2+x)^{m}   = \left[ \frac{(2x+1)^2-1}4 \right]^m
\]
is an even function in $(2x+1)$. Thus, the derivative on the right side of Equation~\eqref{eq:NegFirstBinomIdentity2} is an odd function
in $(2x+1)$ and must vanish when substituting $x= -1/2$. The result is Equation \eqref{eq:NegFirstBinomIdentity}.
\end{proof}

With this, we are ready to prove the main result of this subsection.

\begin{proof}[Proof of Proposition \ref{prop:NegFirstReformulate}]
It is easy to check that for each arbitrary choice of values $\gamma_{2i}$ $(0\le 2i \le -r)$, there is a unique
$f(x) = \sum_{i=0}^{-r}\gamma_i x^i$ satisfying Equation \eqref{eq:DegreeAnyFirst} for $1 \leq m \leq\lceil -r/2\rceil$.
In particular, Equation \eqref{eq:DegreeAnyFirst} describes a vector space $V$ of polynomials with $\dim V = \lfloor -r/2 \rfloor +1$.
Note also that Equation \eqref{eq:NegFirstReformulate} describes a vector space $V'$ of polynomials with $\dim V' = \lfloor -r/2 \rfloor +1$
and basis
\begin{equation}
\label{eq:NegFirstBinomBasis}
    g_s(x) = (x-2)^{-r-2s}x^{2s}, \quad 0 \le s \le \lfloor -r/2 \rfloor.
\end{equation}
Hence, for $V=V'$, it is sufficient to show that each of these polynomials satisfies Equation \eqref{eq:DegreeAnyFirst} for $1 \leq m \leq\lceil -r/2\rceil$.

For
\[
g_s(x) = (x-2)^{-r-2s}x^{2s}=\sum_{j=0}^{-r-2s}{-r-2s \choose j} (-2)^{-r-2s-j} x^{2s+j}=\sum_{i=0}^{-r}\gamma_i x^i
\]
we have
\begin{equation}
\label{eq:NegFirstBinomGamma}
    \gamma_i =
        \begin{cases}
            0,                              & i < 2s,
            \\
            {-r-2s \choose i-2s} (-2)^{-r-i},         & i \ge 2s.
\end{cases}
\end{equation}
Substituting Equation \eqref{eq:NegFirstBinomGamma} into Equation \eqref{eq:DegreeAnyFirst} and applying Lemma \ref{lem:NegFirstBinomIdentity} yields
\begin{align*}
    & \sum\limits_{i=0}^{m} {m\choose i} {1-r\choose m+i}^{-1} \gamma_{m+i-1}\\
    & = \sum\limits_{i=2s-m+1}^{m} {m\choose i} {1-r\choose m+i}^{-1}{-r-2s \choose m+i-2s-1} (-2)^{-r-m-i+1}\\
    & = \sum\limits_{i=2s-m+1}^{m} {m\choose i} \frac{(m+i)!(1-r-m-i)!(-r-2s)!}{(1-r)!(m+i-2s-1)!(-r-m-i+1)!}\: (-2)^{-r-m-i+1}\\
    & = \frac{(-r-2s)!}{(1-r)!}\: (-2)^{-r-m+1} \sum\limits_{i=2s-m+1}^{m} {m\choose i} \frac{(m+i)!}{(m+i-2s-1)!}\: (-2)^{-i} = 0.
    \qedhere
\end{align*}
\end{proof}

Theorem \ref{thrm:Relations} now follows from Corollaries \ref{cor:RelPosEven}
and \ref{cor:RelNegOdd}, Lemmas \ref{lem:Relr1} and \ref{lem:RelPosOddg1}, and Proposition
\ref{prop:NegFirstReformulate}.


\section{Relation to Euler polynomials and Bernoulli numbers}
\label{sec:EulerPoly}

In this section, we give two additional characterizations of elements of
$\mathcal{F}_r$ in terms of their coefficients. The initial observation motivating
these characterizations is the following corollary to Theorem \ref{thrm:Relations},
which is easily proven by induction on $m$ for each case of the parity and sign of $r$.

\begin{corollary}
\label{cor:OddFromEvBasic}
Let $r$ be an integer. For each arbitrary choice $\{\gamma_{2i}\}_{i=0}^\infty \subset \K$,
there is a unique choice of $\gamma_{2i+1}\in\K$ for each $i$ such that
$\varphi(x)=\sum_{i= 0}^\infty \gamma_i x^i\in \mathcal{F}_r$.
\end{corollary}

Hence, for an element of $\mathcal{F}_r$, the $\gamma_{2i+1}$ are completely determined
by the $\gamma_{2i}$. The goal of this section is to indicate explicitly how to compute
the odd-degree coefficients from the even, Theorem \ref{thrm:OddFromEven},
and similarly how to compute the even-degree coefficients from the odd,
Theorem \ref{thrm:EvenFromOdd}. These pleasant relationships are stated most succinctly
in terms of the Bernoulli numbers and the coefficients of the Euler polynomials.
First, we recall the definitions of these constants and establish the properties we will need.

For $n = 0, 1, 2, \ldots$, let $E_n(x)$ denote the $n$th Euler polynomial defined by
the generating function
\begin{equation}
\label{eq:EulerGenerating}
    \frac{2e^{xt}}{e^t + 1}
    =
    \sum\limits_{n=0}^\infty E_n(x) \frac{t^n}{n!}.
\end{equation}
As an obvious consequence of this equation, we have
\begin{equation}
\label{eq:EulerPolySum}
     E_n(x+1)+E_n(x)=2x^n.
\end{equation}
Observe that when $n$ is even (respectively odd), the only nonzero even degree
(respectively odd degree) term in $E_n(x)$ is the leading monomial $x^n$. We
use the notation ${n\brack i}$ to denote the (negatives of the) odd degree
coefficients of $E_{2n}(x)$ and ${n \brace i}$ to denote the
(negatives of the) even degree coefficients of $E_{2n+1}(x)$, i.e
\begin{equation}
\label{eq:EulerDefBracks}
    \sum\limits_{i=1}^n {n\brack i}x^{2i-1} = x^{2n} - E_{2n}(x)
    \quad\mbox{and}\quad
    \sum\limits_{i=0}^n {n\brace i}x^{2i} = x^{2n+1} - E_{2n+1}(x).
\end{equation}
Note that the ${n\brack i}$ are integers while the ${n\brace i}$ are generally
rational, ${n\brack i} = 0$ for $i \leq 0$ or $i > n$, and
${n\brace i} = 0$ for $i < 0$ or $i > n$. Due to this and other differences
between the properties of the ${n\brack i}$ and ${n\brace i}$, it will usually
be simplest to distinguish the even and odd cases with this notation. However,
it will sometimes be convenient to use the \emph{unified notation}
\begin{equation}
\label{unified notation}
    {n \choose i}_r :=
        \begin{cases}
            {r/2+n\brack r/2+i},            &\mbox{$r$ is even and
                                            $n\ge \operatorname{max}\{0,-r/2\}$,}
            \\
            -{-r/2-i\brack -r/2-n},         &\mbox{$r$ is even and
                                            $0\le n<-r/2$,}
            \\
            {(r-1)/2+n\brace (r-1)/2+i},    &\mbox{$r$ is odd and
                                            $n\ge \operatorname{max}\{0,-(r-1)/2\}$,}
            \\
            -{-(r+1)/2-i\brace -(r+1)/2-n}, &\mbox{$r$ is odd and
                                            $0\le n<-(r-1)/2$.}
\end{cases}
\end{equation}

We use $B_n$ to denote the $n$th Bernoulli number, given by the generating function
\[
    \frac{t}{e^t - 1} = \sum\limits_{n=0}^\infty B_n \frac{t^n}{n!}.
\]
The Bernoulli numbers are related to the even Euler polynomial coefficients via
\begin{equation}
\label{eq:EulerBernoulli-Even}
    {n\brack i}
    =    \frac{4^{n-i+1}-1}{n-i+1}B_{2(n-i+1)}{2n\choose 2i-1},
    \quad\quad 0\leq i \leq n.
\end{equation}

By the obvious consequence $\frac{d}{dx} E_{n+1}(x) = (n+1)E_n(x)$ of Equation
\eqref{eq:EulerGenerating}, we have
\begin{equation}
\label{eq:EulerEvenToOdd}
    \frac{2i+1}{2n+2}{n+1\brack i+1} = {n \brace i},
    \quad\quad 0\leq i \leq n,
\end{equation}
and hence
\begin{equation}
\label{eq:EulerBernoulli-Odd}
    {n\brace i}
    =    \frac{4^{n-i+1}-1}{n-i+1}B_{2(n-i+1)}{2n+1\choose 2i},
    \quad\quad 0\leq i \leq n.
\end{equation}
Differentiating twice yields the useful corollary
\begin{equation}
\label{eq:EvenEulerRecursion}
    {n \brack i}
        =   \frac{(2n)(2n-1)}{(2i-1)(2i-2)}{n-1 \brack i-1}
        =   \frac{{2n\choose 2}}{{2i-1\choose 2}}{n-1 \brack i-1},
        \quad\quad 2\leq i \leq n.
\end{equation}

In order to express the unified notation of Equation~\eqref{unified notation} with the help of
Equations \eqref{eq:EulerBernoulli-Even} and \eqref{eq:EulerBernoulli-Odd}, we define
\[
    f_m:= \frac{4^{m}-1}{m} B_{2m}
\]
and observe
\begin{equation}
\label{simplified cases}
    {n \choose i}_r     =   f_{n-i+1}\cdot
        \begin{cases}
            {r+2n\choose r+2i-1},       &\mbox{$n\ge\operatorname{max}\{0,\lceil-r/2\rceil\}$,}
            \\
            -{-r-2i\choose -r-2n-1},    &\mbox{$0\le n<\lceil-r/2\rceil$.}
        \end{cases}
\end{equation}

We will also make use of the following.

\begin{lemma}
\label{lem:EvenFromOdd-SumForm}
For integers $0 \leq j \leq n$, we have
\begin{equation}
\label{eq:EvenFromOdd-SumForm}
    \sum\limits_{i=j}^n {2n \choose 2i} B_{2n-2i} {i\brack j}
    =
    \begin{cases}
        0,  &   j < n, \\
        n,  &   j = n.
    \end{cases}
\end{equation}
\end{lemma}
\begin{proof}
First note that as ${i\brack 0} = 0$ for all $i$, the case $j = 0$ is clear.
Similarly, by Equation \eqref{eq:EulerBernoulli-Even} and the fact that
$B_2 = 1/6$, one checks that ${n\brack n} = n$ for any $n\geq 0$. Then
the case $j = n$ follows from a simple computation.

Now, for $\ell \geq 2$, we have from \cite[Theorem 13.3.6]{Moll} the identity
\begin{equation}
\label{eq:BernMollIdentity}
    \sum\limits_{s=1}^\ell {2\ell\choose 2s} (2^{2s}-1)B_{2s} B_{2(\ell-s)} = 0,
\end{equation}
see also \cite{Mordell}. Multiplying both sides by a constant in terms of $j \geq 1$,
we have
\begin{align*}
    0 &=
        \frac{2(2\ell+2j-2)!}{(2j-1)! (2\ell)!}
        \sum\limits_{s=1}^\ell {2\ell\choose 2s} (2^{2s}-1)B_{2s} B_{2(\ell-s)}
    \\&=
        \sum\limits_{s=1}^\ell \frac{(2\ell+2j-2)!(2s+2j-2)!(4^s-1)}
            {(2s+2j-2)!(2\ell-2s)!(2j-1)! (2s-1)! s} B_{2s} B_{2(\ell-s)}
    \\&=
        \sum\limits_{s=1}^\ell {2\ell+2j-2\choose 2s+2j-2}{2s+2j-2\choose 2j-1}
            \frac{4^s-1}{s} B_{2s} B_{2(\ell-s)}.
\end{align*}
Setting $n = \ell + j - 1 \geq j+1$ and using Equation \eqref{eq:EulerBernoulli-Even}, we continue
\begin{align*}
    &=
        \sum\limits_{s=1}^{n-j+1} {2n\choose 2s+2j-2}{2s+2j-2\choose 2j-1}
            \frac{4^s-1}{s} B_{2s} B_{2(n-s-j+1)}
    \\&=
        \sum\limits_{s=1}^{n-j+1} {2n\choose 2(s+j-1)} {s+j-1\brack j}
            B_{2(n-s-j+1)},
\end{align*}
and then setting $i = s+j-1$ completes the proof.
\end{proof}


\subsection{Odd coefficients from even: Euler polynomials}
\label{subsec:EulerPoly}

In this subsection, we give explicit formulas for the odd coefficients of
an element of $\mathcal{F}_r$ in terms of the even coefficients, yielding
another characterization of elements of $\mathcal{F}_r$ and indicating the
connection with Euler polynomials. This generalizes \cite[Theorem 5.1]{HHS},
which gave the corresponding result for the case $r=0$.
Specifically, we prove the following.

\begin{theorem}
\label{thrm:OddFromEven}
Let $\varphi(x) = \sum_{i= 0}^\infty \gamma_i x^i$ and let $r\in\Z$.
As in Theorem \ref{thrm:Relations}, we stipulate that $\gamma_i = 0$ for $i < 0$. Then
$\varphi(x) \in\mathcal{F}_r$ if and only if the following conditions
are fulfilled.

If $r = 2k$ is even, then for each $n \geq \max\{0,-k\}$,
\begin{equation}
\label{eq:OddFromEven-Even}
    \gamma_{2n+1} = \sum\limits_{i=-k}^n {n+k\brack i+k}\gamma_{2i},
\end{equation}
and, if $r < 0$ so that $k < 0$, then for $0\leq n\leq -k-1$,
\begin{equation}
\label{eq:OddFromEven-EvenFirst}
    \gamma_{2n+1} = \sum\limits_{i=0}^n -{-k-i\brack -k-n}\gamma_{2i}.
\end{equation}

If $r = 2k+1$ is odd, then for each $n \geq \max\{0,-k\}$,
\begin{equation}
\label{eq:OddFromEven-Odd}
    \gamma_{2n+1} = \sum\limits_{i=-k}^{n} {n+k\brace i+k}\gamma_{2i},
\end{equation}
and, if $r < 0$ so that $k < 0$, then for $0\leq n\leq -k-1$,
\begin{equation}
\label{eq:OddFromEven-OddFirst}
    \gamma_{2n+1} = \sum\limits_{i=0}^n -{-k-i-1\brace -k-n-1}\gamma_{2i}.
\end{equation}
\end{theorem}

\begin{remark}
\label{rem:thrmOddFromEvenUnified}
Note that in Equations \eqref{eq:OddFromEven-Even} and \eqref{eq:OddFromEven-Odd},
the sum can be taken over all integers $i$ as the coefficient vanishes unless
$-k \leq i \leq n$. Hence, using the unified notation of Equation \eqref{unified notation},
the conditions expressed by Equations \eqref{eq:OddFromEven-Even}, \eqref{eq:OddFromEven-EvenFirst},
\eqref{eq:OddFromEven-Odd}, and \eqref{eq:OddFromEven-OddFirst} in Theorem \ref{thrm:OddFromEven}
can be stated succinctly as: for every $n \geq 0$,
\begin{equation}
\label{eq:thrmOddFromEvenUnified}
    \gamma_{2n+1} = \sum\limits_{i=0}^n {n \choose i}_r \gamma_{2i}.
\end{equation}
\end{remark}

To prove Theorem \ref{thrm:OddFromEven}, we observe that by Corollary
\ref{cor:OddFromEvBasic}, for any value of $r$, either direction of the
biconditional in Theorem \ref{thrm:OddFromEven} follows from the opposite implication.
Note further that when $r < 0$, just as in the statement of Theorem \ref{thrm:Relations},
the $\gamma_i$ for $i \leq -r$ do not contribute to
Equations \eqref{eq:OddFromEven-Even} and \eqref{eq:OddFromEven-Odd},
while the $\gamma_i$ for $i > -r$ do not appear in Equations
\eqref{eq:OddFromEven-EvenFirst} and \eqref{eq:OddFromEven-OddFirst}.
Hence, as in Section \ref{sec:ProofRelations}, we first assume that in the case $r < 0$,
the $\gamma_i$ vanish for $i \leq -r$. We begin with the case of $r$ even.

\begin{lemma}
\label{lem:OddFromEven:Even}
Let $\varphi(x) = \sum_{i= 0}^\infty \gamma_i x^i \in \K[\![x]\!]$ and let $r=2k$
be an even integer. If $r < 0$, assume that $\gamma_i = 0$ for each $i \leq -r$.
Then the $\gamma_i$ satisfy Equation \eqref{eq:OddFromEven-Even} for
each $n \geq \max\{0,-k\}$ if and only if $\varphi(x) \in \mathcal{F}_r$.
\end{lemma}
\begin{proof}
If $r = 0$, then by an application of Theorem \ref{thrm:Relations}, this result is precisely
\cite[Theorem 5.1]{HHS}. If $r\neq 0$, define
$\overline{\varphi}(x) = x^r\varphi(x) = \sum_{i=0}^\infty \overline{\gamma}_i x^i$.
By Lemma \ref{lem:EvenJustShifted}, if $\varphi(x)\in\mathcal{F}_r$, then
$\overline{\varphi}(x) \in \mathcal{F}_0$, and hence by the above argument for $r=0$,
we have that for each $n \geq 0$,
\[
    \overline{\gamma}_{2n+1}
    =   \sum\limits_{i=0}^n {n\brack i}\overline{\gamma}_{2i}.
\]
The result then follows from Corollary \ref{cor:OddFromEvBasic} and the fact that
$\overline{\gamma}_i = \gamma_{i-2k}$ for each $i \geq 0$.
\end{proof}

\begin{lemma}
\label{lem:OddFromEven:Odd}
Let $\varphi(x) = \sum_{i= 0}^\infty \gamma_i x^i \in \K[\![x]\!]$ and let $r=2k+1$
be an odd integer. If $r < 0$, assume that $\gamma_i = 0$ for each $i \leq -r$.
Then the $\gamma_i$ satisfy Equation \eqref{eq:OddFromEven-Odd} for
each $n \geq \max\{0,-k\}$ if and only if $\varphi(x) \in \mathcal{F}_r$.
\end{lemma}
\begin{proof}
We recall \cite[Equation (5.7)]{HHS},
\begin{equation}
\label{eq:EulerEq5.7}
    \sum\limits_{\substack{i=0 \\ \mbox{\scriptsize $m+i$ even}}}^m (-1)^i {m \choose i}
        \sum\limits_j {\frac{m+i}{2}\brack j} x^{2j-1}
        +
        \sum\limits_{\substack{i=0 \\ \mbox{\scriptsize $m+i$ odd}}}^m (-1)^i {m \choose i} x^{m+i}
    = 0,
\end{equation}
which was derived using an application of \cite[Theorem 7.4]{GesselUmbral}.
For $r = 1$, we differentiate Equation \eqref{eq:EulerEq5.7} once and apply
Equation \eqref{eq:EulerEvenToOdd}, yielding
\begin{align*}
    \sum\limits_{\substack{i=0 \\ \mbox{\scriptsize $m+i$ even}}}^m (-1)^i {m \choose i}
        (m+i) \sum\limits_j {\frac{m+i}{2}-1\brace j-1} x^{2j-2}
        \\ \quad\quad +
        \sum\limits_{\substack{i=0 \\ \mbox{\scriptsize $m+i$ odd}}}^m (-1)^i {m \choose i}
       (m+i) x^{m+i-1}
    &= 0.
\end{align*}
We interpret $x$ as an umbral variable \cite{GesselUmbral} and define the
functional $\Gamma\co\K[x]\to\K$ by $\Gamma(x^i) = \gamma_{i}$. Applying $\Gamma$ to the last
equation results in
\begin{align*}
    \sum\limits_{\substack{i=0 \\ \mbox{\scriptsize $m+i$ even}}}^m (-1)^i {m \choose i}
        (m+i) \sum\limits_j {\frac{m+i}{2}-1\brace j-1} \gamma_{2j-2}
        \\ \quad\quad +
        \sum\limits_{\substack{i=0 \\ \mbox{\scriptsize $m+i$ odd}}}^m (-1)^i {m \choose i}
       (m+i) \gamma_{m+i-1}
    &= 0.
\end{align*}
Using the fact that the $\gamma_i$ satisfy Equation \eqref{eq:OddFromEven-Odd} for $k=0$,
we rewrite the sum over $j$ as
\[
    \sum\limits_j {\frac{m+i}{2}-1\brace j-1} \gamma_{2j-2}
    =
    \sum\limits_j {\frac{m+i}{2}-1\brace j} \gamma_{2j}
    =
    \gamma_{m+i-1},
\]
yielding
\[
    \sum\limits_{i=0}^m (-1)^i {m \choose i}
       (m+i) \gamma_{m+i-1} =
    \sum\limits_{i=0}^m (-1)^i {m \choose i}
        \frac{(m+i)!}{(m+i-1)!} \gamma_{m+i-1}
    = 0.
\]
This is Equation \eqref{eq:RelPosOddKrewrite} for $k=0$, which was demonstrated
in Lemma \ref{lem:RelPosOddCharacterize} to be equivalent to Equation
\eqref{eq:DegreePositiveOdd}, and hence by Theorem \ref{thrm:Relations}
and Corollary \ref{cor:OddFromEvBasic}, the result holds for $r = 1$.

Now assume $r \not= 1$.
By Lemma \ref{lem:OddJustShifted},
$\overline{\varphi}(x) = x^{r-1}\varphi(x)  = \sum_{i=0}^\infty \overline{\gamma}_i x^i$ is an element
of $\mathcal{F}_1$ if and only if $\varphi(x)\in\mathcal{F}_r$. By the above argument for
the case $r = 1$, we have that for each $n\ge 0$
\[
    \overline{\gamma}_{2n+1} = \sum\limits_{i=0}^{n} {n\brace i}\overline{\gamma}_{2i},
\]
which, applying $\overline{\gamma}_i = \gamma_{i-2k}$, is equivalent to
Equation \eqref{eq:OddFromEven-Odd} for the $\gamma_i$.
\end{proof}

We now consider the case of the first $1-r$ terms when $r < 0$. Similar to the proof
of Proposition \ref{prop:NegFirstReformulate}, we start with an auxiliary result.

\begin{lemma}
\label{lem:OddFromEven:FirstAux}
For $n\geq 0$, we have
\begin{equation}
\label{eq:NegFirstBinomIdentityAux}
    \sum\limits_{i=0}^{n} -{n \brace i}(-2)^{-2i} = (-2)^{-2n-1}
\end{equation}
\end{lemma}
\begin{proof}
Substituting $x= -1/2$ into Equations \eqref{eq:EulerPolySum} and \eqref{eq:EulerDefBracks} yields
\begin{align*}
    \sum\limits_{i=0}^{n} -{n \brace i}(-2)^{-2i}
    &= - \frac 12 \left[ \sum\limits_{i=0}^{n} {n \brace i} \left( \frac 12 \right)^{2i}
    + \sum\limits_{i=0}^{n} {n \brace i}\left(  -\frac 12 \right)^{2i} \right]
    \\&=
    - \frac 12 \left[ \left( \frac 12 \right)^{2n+1} - E_{2n+1}\left( \frac 12 \right)
    + \left( -\frac 12 \right)^{2n+1} - E_{2n+1}\left( -\frac 12 \right) \right]
    \\&=
    \frac 12 \left[ E_{2n+1}\left( -\frac 12 +1 \right)
    + E_{2n+1}\left( -\frac 12 \right) \right] = \left( -\frac 12 \right)^{2n+1}.
    \qedhere
\end{align*}
\end{proof}

With this, we are ready to prove the following.
\begin{lemma}
\label{lem:OddFromEven:First}
Assume $r < 0$ and let $f(x) = \sum_{i=0}^{-r}\gamma_i x^i$.
Then for $0\leq n \leq -k-1$, the $\gamma_i$ satisfy Equation \eqref{eq:OddFromEven-EvenFirst}
when $r=2k$ is even and Equation \eqref{eq:OddFromEven-OddFirst} when $r=2k+1$ is odd
if and only if $f(x)\in\mathcal{F}_r$.
\end{lemma}
\begin{proof}
Using the description of $\mathcal{F}_r$ given by Theorems \ref{thrm:Deg0HHS} and \ref{thrm:LaurentPoly}
and the generator given in Equation \eqref{eq:xox-2sqr}, $f(x)\in\mathcal{F}_r$ if and only if we can
express
\begin{equation}
\label{eq:NegFirstReformulate2}
    f(x)    =   (x - 2)^{-r} \sum\limits_{i=0}^{\lfloor -r/2 \rfloor} \delta_i
                                \left(\frac{x}{x-2}\right)^{2i}
            =   x^{-r} \sum\limits_{i=0}^{\lfloor -r/2 \rfloor} \delta_i
                                \left(\frac{x-2}{x}\right)^{-r-2i}.
\end{equation}
We will follow the proof of Proposition \ref{prop:NegFirstReformulate} and consider two cases.

If $r=2k$ is even, it is easy to check that for each arbitrary choice of values $\gamma_{2i}$ $(0\le i \le -k)$, there is a unique
$f(x) = \sum_{i=0}^{-r}\gamma_i x^i$ satisfying Equation \eqref{eq:OddFromEven-EvenFirst} for $0\leq n \leq -k-1$.
In particular, Equation \eqref{eq:OddFromEven-EvenFirst} describes a vector space $V$ of polynomials with $\dim V = -k+1$.
Note also that Equation \eqref{eq:NegFirstReformulate2} describes a vector space $V'$ of polynomials with $\dim V' = -k+1$
and basis
\[
    g_s(x) = (x-2)^{-r-2s}x^{2s}, \quad 0 \le s \le -k.
\]
Hence, for $V=V'$, it is sufficient to show that each of these polynomials satisfies Equation \eqref{eq:OddFromEven-EvenFirst}
for $0\leq n \leq -k-1$.

For $n<s$, Equation \eqref{eq:OddFromEven-EvenFirst} is trivially satisfied as all entries $\gamma_i =0$,
cf. Equation~\eqref{eq:NegFirstBinomGamma}. For $s \leq n \leq -k-1$, using
Equations \eqref{eq:NegFirstBinomGamma}, \eqref{eq:EulerBernoulli-Even}, \eqref{eq:EulerBernoulli-Odd} and \eqref{eq:NegFirstBinomIdentityAux},
and recalling $r=2k$, the right side of Equation \eqref{eq:OddFromEven-EvenFirst} can be written
\begin{align*}
    & \sum\limits_{i=0}^n -{-k-i\brack -k-n}\gamma_{2i} =
    \sum\limits_{i=0}^n -\frac{4^{n-i+1}-1}{n-i+1}B_{2(n-i+1)}{-2k-2i \choose -2k-2n-1}\gamma_{2i}
    \\&=
    \sum\limits_{i=s}^n -\frac{4^{n-i+1}-1}{n-i+1}B_{2(n-i+1)}{-2k-2i \choose -2k-2n-1}{-2k-2s \choose 2i-2s} (-2)^{-2k-2i}
    \\&=
    \sum\limits_{i=s}^n -\frac{4^{n-i+1}-1}{n-i+1}B_{2(n-i+1)}
    \frac{(-2k-2i)!(-2k-2s)! (-2)^{-2k-2i}}{(-2k-2n-1)!(2n-2i+1)!(2i-2s)!(-2k-2i)!}
    \\&=
    \frac{(-2k-2s)! (-2)^{-2k-2s}}{(2n-2s+1)!(-2k-2n-1)!} \sum\limits_{i=s}^n -\frac{4^{n-i+1}-1}{n-i+1}B_{2(n-i+1)}
    \frac{(2n-2s+1)! (-2)^{2s-2i}}{(2i-2s)!(2n-2i+1)!}
    \\&=
    {-2k-2s \choose 2n-2s+1} (-2)^{-2k-2s} \sum\limits_{i=s}^n -\frac{4^{n-i+1}-1}{n-i+1} B_{2(n-i+1)} {2n-2s+1 \choose 2i-2s}
    (-2)^{2s-2i}
    \\&=
    {-2k-2s \choose 2n-2s+1} (-2)^{-2k-2s} \sum\limits_{i=s}^n -{n-s \brace i-s}(-2)^{2s-2i}
    \\&=
    {-2k-2s \choose 2n-2s+1} (-2)^{-2k-2s} \sum\limits_{i=0}^{n-s} -{n-s \brace i}(-2)^{-2i}
    \\&=
    {-2k-2s \choose 2n-2s+1} (-2)^{-2k-2s} (-2)^{-2(n-s)-1} = {-2k-2s \choose 2n-2s+1} (-2)^{-2k-2n-1}
    \\&=\gamma_{2n+1}.
\end{align*}

If $r=2k+1$ is odd, we have $\dim V = \dim V'= -k$ and need to check that
\[
    g_s(x) = (x-2)^{-r-2s}x^{2s}, \quad 0 \le s \le -k-1
\]
satisfies Equation \eqref{eq:OddFromEven-OddFirst} for $0\leq n \leq -k-1$.

For $s \leq n \leq -k-1$, using
Equations \eqref{eq:NegFirstBinomGamma}, \eqref{eq:EulerBernoulli-Odd} and \eqref{eq:NegFirstBinomIdentityAux},
and recalling $r=2k+1$, the right side of Equation \eqref{eq:OddFromEven-OddFirst} can be written
\begin{align*}
    & \sum\limits_{i=0}^n -{-k-i-1\brace -k-n-1}\gamma_{2i} =
    \sum\limits_{i=0}^n -\frac{4^{n-i+1}-1}{n-i+1}B_{2(n-i+1)}{-2k-2i-1 \choose -2k-2n-2}\gamma_{2i}
    \\&=
    \sum\limits_{i=s}^n -\frac{4^{n-i+1}-1}{n-i+1}B_{2(n-i+1)}{-2k-2i-1 \choose -2k-2n-2}{-2k-2s-1 \choose 2i-2s} (-2)^{-2k-2i-1}
    \\&=
    {-2k-2s-1 \choose 2n-2s+1} (-2)^{-2k-2s-1} \sum\limits_{i=s}^n -\frac{4^{n-i+1}-1}{n-i+1} B_{2(n-i+1)} {2n-2s+1 \choose 2i-2s}
    (-2)^{2s-2i}
    \\&=
    {-2k-2s-1 \choose 2n-2s+1} (-2)^{-2k-2s-1} (-2)^{-2(n-s)-1} = {-2k-2s-1 \choose 2n-2s+1} (-2)^{-2k-2n-2}
    \\&=\gamma_{2n+1}.
    \qedhere
\end{align*}
\end{proof}

Theorem \ref{thrm:OddFromEven} now follows from Lemmas \ref{lem:OddFromEven:Even},
\ref{lem:OddFromEven:Odd}, and \ref{lem:OddFromEven:First}.


\subsection{Even coefficients from odd: Bernoulli numbers}
\label{subsec:EvenFromOdd}

In this subsection, we give another characterization of elements of $\mathcal{F}_r$
by exploring whether the even degree coefficients are determined by the odd. The
results in this direction are not as straightforward as those of Theorem \ref{thrm:OddFromEven};
there are for some values of $r$ an additional constraint and an unconstrained
coefficient. However, we again find an appealing characterization, this time most
easily expressed in terms of the Bernoulli numbers. Specifically, we have the following.

\begin{theorem}
\label{thrm:EvenFromOdd}
Let $\varphi(x) = \sum_{i= 0}^\infty \gamma_i x^i$ and let $r\in\Z$.
As above, we stipulate that $\gamma_i = 0$ for $i < 0$. Then
$\varphi(x) \in\mathcal{F}_r$ if and only if the following conditions
are fulfilled.

If $r = 2k$ is even, then for each $n \geq \max\{0,1-k\}$
\begin{equation}
\label{eq:EvenFromOdd-Main}
    \gamma_{2n} = \frac{2}{2n+r}\sum\limits_{i=-k}^n {2n+r \choose 2i+r} B_{2n-2i} \gamma_{2i+1},
\end{equation}
and, if $r \leq 0$, then $\gamma_{-r}$ is unconstrained, $\gamma_{1-r} = 0$, and (when $r < 0$)
\begin{equation}
\label{eq:EvenFromOdd-First}
    \gamma_{2n} = \sum\limits_{i=0}^n \frac{2}{2i+r} {-2i-r \choose -2n-r} B_{2n-2i} \gamma_{2i+1}
\end{equation}
holds for $0 \leq n \leq -k-1$. This completely determines the $\gamma_{2i}$ in terms of the $\gamma_{2i+1}$.

If $r = 2k+1$ is odd, then Equation \eqref{eq:EvenFromOdd-Main} holds for each $n \geq \max\{0,-k\}$,
and, if $r < 0$, then Equation \eqref{eq:EvenFromOdd-First} holds for $0 \leq n \leq -k-1$.
This completely determines the $\gamma_{2i}$ in terms of the $\gamma_{2i+1}$.
\end{theorem}

\begin{proof}
We first observe that, by an induction argument similar to that described in Corollary
\ref{cor:OddFromEvBasic}, an arbitrary choice of $\{\gamma_{2i+1}\}_{i=0}^\infty\subset\K$
uniquely determines an element of $\mathcal{F}_r$ except when $r = 2k \leq 0$, in which case
any choice of $\{\gamma_{2i+1}\}_{i=0}^\infty\subset\K$ such that $\gamma_{1-r}=0$ along with
an arbitrary choice of $\gamma_{-r}$ uniquely determines an element of $\mathcal{F}_r$.
Hence, it is sufficient to show that the coefficients of any element of $\mathcal{F}_r$
satisfy Equations \eqref{eq:EvenFromOdd-Main} and \eqref{eq:EvenFromOdd-First} as described
in the statement of the theorem.

So assume $\varphi(x)=\sum_{i= 0}^\infty \gamma_i x^i\in \mathcal{F}_r$
for some $r$. The idea of the proof in each case is to use the characterization given in
Theorem \ref{thrm:OddFromEven} to rewrite the right side of Equation \eqref{eq:EvenFromOdd-Main}
or \eqref{eq:EvenFromOdd-First} in a form that can be simplified using
Lemma \ref{lem:EvenFromOdd-SumForm}.

Assume $r = 0$ and then, using $\gamma_1 = 0$ and Equation \eqref{eq:OddFromEven-Even},
the right side of Equation \eqref{eq:EvenFromOdd-Main} with $n \geq 1$ becomes
\begin{align*}
    \frac{1}{n}\sum\limits_{i=1}^n {2n \choose 2i} B_{2n-2i} \gamma_{2i+1}
    &=
    \frac{1}{n}\sum\limits_{i=1}^n {2n \choose 2i} B_{2n-2i}
        \sum\limits_{j=0}^{i} {i\brack j}\gamma_{2j}
    \\&=
    \frac{1}{n} \sum\limits_{j=0}^{n} \gamma_{2j}
        \sum\limits_{i=j}^n {2n \choose 2i} B_{2n-2i} {i\brack j}.
\end{align*}
Applying Lemma \ref{lem:EvenFromOdd-SumForm} to the sum over $i$,
this expression is equal to $\gamma_{2n}$.

If $r = 2k > 0$, then
$\overline{\varphi}(x) = x^r\varphi(x) = \sum_{i=0}^\infty \overline{\gamma}_i x^i\in\mathcal{F}_0$
by Lemma \ref{lem:EvenJustShifted}. Hence by the above argument, the $\overline{\gamma}_i$ satisfy
Equation \eqref{eq:EvenFromOdd-Main} for $\mathcal{F}_0$
and any $n \geq 1$, and hence using $\overline{\gamma}_i = \gamma_{i-2k}$, the $\gamma_i$
satisfy Equation \eqref{eq:EvenFromOdd-Main} for $\mathcal{F}_r$ and any $n \geq 0$.
If $r = 2k < 0$, then we apply the same argument assuming $\gamma_0=\gamma_1=\ldots=\gamma_{-r}=0$
to see that Equation \eqref{eq:EvenFromOdd-Main} for $\mathcal{F}_r$ is satisfied for any
$n \geq 1-k$.

Now assume $r = 1$. When $n = 0$, Equation \eqref{eq:EvenFromOdd-Main} simply states
$\gamma_0 = 2\gamma_1$, which, as ${0\brace 0} = 1/2$, is equivalent to the same case of
Equation \eqref{eq:OddFromEven-Odd}. So assume $n\geq 1$, and then we use Equation
\eqref{eq:OddFromEven-Odd} to express the right side of Equation \eqref{eq:EvenFromOdd-Main} as
\[
    \frac{2}{2n+1}\sum\limits_{i=0}^n {2n+1 \choose 2i+1} B_{2n-2i} \gamma_{2i+1}
    =
    \frac{2}{2n+1}\sum\limits_{i=0}^n {2n+1 \choose 2i+1} B_{2n-2i}
        \sum\limits_{j=0}^i {i\brace j}\gamma_{2j}.
\]
Applying Equation \eqref{eq:EulerEvenToOdd} to rewrite ${i\brace j}$ and exchanging the
sums yields
\begin{equation}
\label{eq:EvenToOddStepR1}
    \frac{2}{2n+1}\sum\limits_{j=0}^n \gamma_{2j}
        \sum\limits_{i=j}^n {2n+1 \choose 2i+1} B_{2n-2i}
        \frac{2j+1}{2i+2} {i+1\brack j+1}.
\end{equation}
The coefficient of $\gamma_0$, corresponding to $j=0$, can be rewritten as
\[
    \frac{2(2n)!}{(2n+2)!} \sum\limits_{i=0}^n {2(n+1)\choose 2(i+1)} B_{2n-2i}
        {i+1\brack 1},
\]
and this sum vanishes by Lemma \ref{lem:EvenFromOdd-SumForm}. Using this fact and Equation
\eqref{eq:EvenEulerRecursion}, we rewrite \eqref{eq:EvenToOddStepR1} as
\[
    2\sum\limits_{j=1}^n \gamma_{2j}
        \sum\limits_{i=j}^n \frac{{2i+2\choose 2}(2n)!(2j+1)}{{2j+1\choose 2}(2i+2)!(2n-2i)!} B_{2n-2i}
        {i\brack j}
    =
    \sum\limits_{j=1}^n \frac{\gamma_{2j}}{j}
        \sum\limits_{i=j}^n {2n\choose 2i } B_{2n-2i}
        {i\brack j}.
\]
Again by Lemma \ref{lem:EvenFromOdd-SumForm}, this is equal to $\gamma_{2n}$.

If $r = 2k+1 > 1$, then
$\overline{\varphi}(x) = x^{r-1}\varphi(x) = \sum_{i=0}^\infty \overline{\gamma}_i x^i\in\mathcal{F}_1$
by Lemma \ref{lem:OddJustShifted}. Hence by the above argument, the $\overline{\gamma}_i$ satisfy
Equation \eqref{eq:EvenFromOdd-Main} for $\mathcal{F}_1$
and any $n \geq 0$, and hence using $\overline{\gamma}_i = \gamma_{i-2k}$, the $\gamma_i$
satisfy Equation \eqref{eq:EvenFromOdd-Main} for $\mathcal{F}_r$ and any $n \geq 0$.
If $r = 2k+1 < 0$, then we apply the same argument assuming $\gamma_0=\gamma_1=\ldots=\gamma_{-r}=0$
to see that Equation \eqref{eq:EvenFromOdd-Main} for $\mathcal{F}_r$ is satisfied for any
$n \geq -k$.

With this, we have that every element of $\mathcal{F}_r$ satisfies
Equation \eqref{eq:EvenFromOdd-Main} for the appropriate values of $n$.

Now, we turn to Equation \eqref{eq:EvenFromOdd-First}. First assume $r = 2k < 0$ is even.
For $0 \leq n \leq -k-1$, using Equation \eqref{eq:OddFromEven-EvenFirst}, the right side
of Equation \eqref{eq:EvenFromOdd-First} can be written
\[
    \sum\limits_{i=0}^n \frac{1}{i+k} {-2i-2k \choose -2n-2k} B_{2n-2i}
        \sum\limits_{j=0}^i -{-k-j\brack -k-i}\gamma_{2j}.
\]
Exchanging the sums and now using Equation \eqref{eq:EulerBernoulli-Even} to rewrite
${-k-j\brack -k-i}$ yields
\[
    - \sum\limits_{j=0}^n \gamma_{2j}
        \sum\limits_{i=j}^n \frac{1}{i+k} {-2i-2k \choose -2n-2k}
        {-2k-2j\choose -2k-2i-1}
        \frac{4^{i-j+1}-1}{i-j+1} B_{2(n-i)} B_{2(i-j+1)}.
\]
Expanding the binomial coefficients, simplifying, setting $\ell=n-j+1$ and
$s=i-j+1$, and recalling $r = 2k$, this expression is rewritten as
\begin{equation}
\label{eq:EvenToOddStepFirst}
    \sum\limits_{j=0}^n \frac{4(-r-2j)! \gamma_{2j}}{(-r-2n)!(2n-2j+2)!}
        \sum\limits_{s=1}^\ell {2\ell \choose 2s} (2^{2s}-1)
        B_{2(\ell-s)} B_{2s}.
\end{equation}
By Equation \eqref{eq:BernMollIdentity}, the sum over $s$ vanishes whenever
$\ell \geq 2$, implying that the only nonzero term occurs
when $\ell = 1$, i.e. $j = n$. Then as $B_0 = 1$ and $B_2 = 1/6$, this expression then simplifies
to $\gamma_{2n}$, completing the proof of Equation \eqref{eq:EvenFromOdd-First} when
$r$ is even.

Finally, suppose $r = 2k+1 < 0$ is odd. For $0 \leq n \leq -k-1$,
the right side of Equation \eqref{eq:EvenFromOdd-First} can be expressed using
Equation \eqref{eq:OddFromEven-OddFirst} as
\[
    \sum\limits_{i=0}^n \frac{2}{2i+2k+1} {-2i-2k-1 \choose -2n-2k-1} B_{2n-2i}
        \sum\limits_{j=0}^i -{-k-j-1\brace -k-i-1}\gamma_{2j}.
\]
Following the same steps as in the case of even $r$ except using
Equation \eqref{eq:EulerBernoulli-Odd} to rewrite ${-k-j-1\brace -k-i-1}$, we again
obtain Equation \eqref{eq:EvenToOddStepFirst}, and hence that this expression is
equal to $\gamma_{2n}$.
\end{proof}


\section{Identities for the Bernoulli numbers}
\label{sec:Identities}

In this section, we use the results of Section \ref{sec:EulerPoly} to derive two families
of combinatorial identities. This is achieved by using the $\Z$-graded algebra structure
of $\mathcal{F}$ (cf. Proposition \ref{prop:GradedAlgebra})
combined first with the formulation of Theorem \ref{thrm:OddFromEven}
given in Equation \eqref{eq:thrmOddFromEvenUnified}, then with Theorem \ref{thrm:EvenFromOdd}.

In Subsection \ref{subsec:Identities1}, we derive a family of cubic relations
for the coefficients ${n \choose i}_r$ defined in Equation \eqref{simplified cases},
see Theorem \ref{thrm:Identities1} below. In Subsection \ref{subsec:Identities2},
we derive a family of binomial relations for the Bernoulli numbers, see
Theorem \ref{thrm:Identities2}.


\subsection{Identities derived from Theorem \ref{thrm:OddFromEven}}
\label{subsec:Identities1}

Let $\varphi(x):=\sum_{i=0}^\infty\varphi_i\:x^i\in \mathcal F_r$ and $\psi(x):=\sum_{j=0}^\infty\psi_j\:x^j\in \mathcal F_s$.
By Theorem \ref{thrm:OddFromEven}, we have that for each $n\ge 0$,
\[
    \varphi_{2n+1}=\sum_{i=0}^n {n\choose i}_r\varphi_{2i},
    \quad\quad
    \psi_{2n+1}=\sum_{i=0}^n {n\choose i}_s\psi_{2i}.
\]
The product $(\varphi\cdot \psi)(x):=\sum_{i=0}^\infty\vartheta_i\:x^i$ is in $\mathcal F_{r+s}$ by Proposition \ref{prop:GradedAlgebra}.
On the one hand, we have by the Cauchy product formula that
\begin{align*}
    \vartheta_{2n+1}
        &=  \sum_{\ell =0}^n(\varphi_{2\ell+1}\:\psi_{2(n-\ell)}+\varphi_{2(n-\ell)}\:\psi_{2\ell+1})
        \\
        &=  \sum_{\ell=0}^n\sum_{k=0}^\ell \left({\ell \choose k}_r \varphi_{2k}\:\psi_{2(n-\ell)}
                + {\ell \choose k}_s \varphi_{2(n-\ell)}\:\psi_{2k}\right)
        \\
        &=  \sum_{\substack\alpha,\beta\ge 0}\left({n -\beta \choose \alpha}_r
                +{n-\alpha\choose \beta}_s\right)\varphi_{2\alpha}\:\psi_{2\beta}.
\end{align*}
On the other hand, $(\varphi\cdot \psi)(x)\in\mathcal F_{r+s}$ and
Equation \eqref{eq:thrmOddFromEvenUnified} imply
\begin{align*}
    & \vartheta_{2n+1}
        =  \sum_{k=0}^n{n\choose k}_{r+s}\vartheta_{2k}=\sum_{k=0}^n\sum_{\ell= 0}^k {n\choose k}_{r+s}
                \left(\varphi_{2\ell}\:\psi_{2(k-\ell)}+\varphi_{2\ell+1}\:\psi_{2(k-\ell)-1}\right)
        \\
        & =  \sum_{\substack{\alpha,\beta\ge 0}}{n\choose \alpha+\beta}_{r+s}
                \varphi_{2\alpha}\:\psi_{2\beta}+\sum_{k=0}^n
                \sum_{\substack{\ell,m\ge0\\\ell+m=k-1}}
                \sum_{\alpha=0}^\ell\sum_{\beta=0}^m {n \choose k}_{r+s}{\ell\choose\alpha}_r
                    {m\choose \beta}_s \varphi_{2\alpha}\:\psi_{2\beta}.
\end{align*}
Noticing that we can choose the even coefficients $\varphi_{2\alpha}$ and $\psi_{2\beta}$ freely,
we compare coefficients and derive the following.

\begin{theorem}
\label{thrm:Identities1}
For $n,\alpha,\beta\ge 0$ and $r,s\in \mathbb Z$ the following identity holds
\begin{align}
\label{eq:cubic identity}
    {n -\beta \choose \alpha}_r+{n-\alpha\choose \beta}_s
    =
    {n \choose \alpha+\beta}_{r+s}+\sum_k\sum_{\substack{\ell,m\ge 0,\\ \ell+m=k-1}}
        {n \choose k}_{r+s}{\ell \choose \alpha}_r{m\choose \beta}_s.
\end{align}
\end{theorem}

In view of Equation \eqref{simplified cases}, the above relations can be seen as cubic relations
for the even indexed Bernoulli numbers. These generalize relations encountered in
\cite[Corollary 5.2]{HHS} for the case $r=s=0$.


\subsection{Identities derived from Theorem \ref{thrm:EvenFromOdd}}
\label{subsec:Identities2}

In a similar fashion, we now derive identities from Theorem \ref{thrm:EvenFromOdd}.
To this end we introduce the following notation.  For $n \ne-r/2$ and $0\le i\le n$,
we define
\begin{align}
\label{eq:unified notation even from odd}
    {n\brack i}_r:=
        \begin{cases}
        \frac{2}{2n+r}{2n+r\choose 2i+r}B_{2(n-i)},
            &   n\ge\operatorname{max}\{0,\lfloor-r/2\rfloor+1\},
        \\
        \frac{2}{2i+r}{-2i-r\choose -2n-r}B_{2(n-i)},
            &   0\le n\le\lfloor-(r+1)/2\rfloor.
\end{cases}
\end{align}
With this, we can restate Theorem \ref{thrm:EvenFromOdd} as follows: A power series
$\varphi(x)=\sum_{i=0}^\infty\gamma_i\,x^i$ is an element of $\mathcal F_r$ if and only if
\begin{equation}
\gamma_{2n}=\sum_{i=0}^n
{n \brack i}_r \gamma_{2i+1} \quad \mbox{for all }n\ge 0\mbox{ with }n\ne -r/2
\end{equation}
and, in the case $r\le 0$ is even, $\gamma_{1-r}=0$.

Again we investigate the fact that the product $(\varphi\cdot\psi)(x)=\sum_{i=0}^\infty\vartheta_i\,x^i$
of two series $\varphi(x):=\sum_{i=0}^\infty\varphi_i\,x^i\in \mathcal F_r$ and
$\psi(x):=\sum_{j=0}^\infty\psi_j\,x^j\in \mathcal F_s$ is an element of $\mathcal F_{r+s}$.
For simplicity, we assume that $r$ and $s$ are both elements of the set of positive or odd integers.
We observe that
\begin{align*}
    \vartheta_{2n}
        &=  \sum_{\substack{\ell,m\ge 0,\\ \ell+m=n-1}}
            \varphi_{2\ell+1}\,\psi_{2m+1}+\sum_{\substack{\ell,m\ge 0,\\ \ell+m=n}}\varphi_{2\ell}\,\psi_{2m}
        \\
        &=  \sum_{\substack{\ell,m\ge 0,\\ \ell+m=n-1}}\varphi_{2\ell+1}\,\psi_{2m+1}
                +\sum_{\substack{\ell,m\ge 0,\\ \ell+m=n}}\sum_{i=0}^\ell\sum_{j=0}^m{\ell\brack i}_r
                    {m\brack j}_s\varphi_{2i+1}\,\psi_{2j+1}
        \\
        &=  \sum_{\alpha,\beta\ge0}\varphi_{2\alpha+1}\,\psi_{2\beta+1}
                \left(\delta_{\alpha+\beta,n-1} + \sum_{\substack{\ell\ge \alpha,m\ge \beta,\\ \ell+m=n}}
                {\ell\brack\alpha}_r{m\brack \beta}_s\right),
\end{align*}
where in the last line, we use the Kronecker delta $\delta_{i,j}$, i.e.
$\delta_{i,j}=1$ if $i=j$ and $0$ otherwise.

On the other hand, for $n \neq -(r+s)/2$
\begin{align*}
    \vartheta_{2n}
    &=  \sum_{k=0}^n{n\brack k}_{r+s}\vartheta_{2k+1}
    \\
    &=  \sum_{k=0}^n\sum_{\ell=0}^k{n\brack k}_{r+s}
         \left( \varphi_{2\ell}\,\psi_{2(k-\ell)+1}+\varphi_{2(k-\ell)+1}\,\psi_{2\ell} \right)
    \\
    &=  \sum_{k=0}^n\sum_{\ell=0}^k\sum_{m=0}^\ell{n\brack k}_{r+s}
            \left( {\ell\brack m}_{r}\varphi_{2m+1}\psi_{2(k-\ell)+1}+{\ell\brack m}_{s}
            \varphi_{2(k-\ell)+1}\psi_{2m+1} \right)
    \\
    &=  \sum_{\alpha,\beta\ge 0}\varphi_{2\alpha+1}\psi_{2\beta+1}
            \sum_{k=0}^n{n\brack k}_{r+s}\left( {k-\beta \brack \alpha}_{r}+{k-\alpha\brack \beta}_{s}\right).
\end{align*}
Comparing these two expressions yields the following.

\begin{theorem}
\label{thrm:Identities2}
Let $r$ and $s$ be elements of the set of positive or odd integers.
Then for every $n\ge 0$ with $n \neq -(r+s)/2$ and integers $\alpha,\beta$
such that $0\le \alpha,\beta\le n$, we have the identity
\begin{align}
    \delta_{\alpha+\beta,n-1} + \sum_{\substack{\ell\ge \alpha,m\ge \beta,\\ \ell+m=n}}
    {\ell\brack\alpha}_r{m\brack \beta}_s
    =
    \sum_{k=0}^n{n\brack k}_{r+s}\left( {k-\beta \brack \alpha}_{r}+{k-\alpha\brack \beta}_{s}\right).
\end{align}
\end{theorem}

In view of Equation \eqref{eq:unified notation even from odd}, the above relations can be
interpreted as quadratic identities for the even indexed Bernoulli numbers.
Note that when either $r$ or $s$ is both nonpositive and even, the existence of unconstrained
even indexed coefficients in Theorem \ref{thrm:EvenFromOdd} affects carrying out a similar derivation.


\section{Power sum formulas revisited}
\label{sec:PowerSum}

In this section, we indicate connections between Theorems \ref{thrm:Relations}, \ref{thrm:LaurentPoly},
\ref{thrm:OddFromEven}, and \ref{thrm:EvenFromOdd} and the work of Gould \cite{Gould}
on power sum identities, clarified by Carlitz \cite{Carlitz}. In particular, we demonstrate
that many of the main results of \cite{Gould} can be recovered as special cases of Theorems
\ref{thrm:OddFromEven} and \ref{thrm:EvenFromOdd}.

In \cite{Gould}, Gould investigated power sums of the form
\[
    \sum\limits_{k=0}^n k^p f_n(k)
\]
where $p \geq 0$ is an integer and, for integers $0 \leq k \leq n$,
the $f_n(k) \in \K$ form a triangular array that is symmetric in the
sense that $f_n(k) = f_n(n-k)$.

To relate Gould's results to those contained here, define
\begin{equation}
\label{eq:G-CSnpDef}
    S_{n,p} := \frac{1}{n^p} \sum\limits_{k=0}^n k^p f_n(k)
    =   \sum\limits_{k=0}^n \frac{k^p}{n^p} f_n(k),
\end{equation}
and consider the generating function
\[
    \varphi_n(x)
    :=  \sum\limits_{p=0}^\infty S_{n,p} x^p
    =   \sum\limits_{k=0}^n f_n(k) \sum\limits_{p=0}^\infty \frac{k^p}{n^p} x^p
    =   \sum\limits_{k=0}^n \frac{f_n(k)}{1 - \frac{k}{n} x}.
\]
We claim that $\varphi_n(x) \in \calF_1$. To see this, we use the symmetry
$f_n(k) = f_n(n-k)$ to compute
\begin{align*}
    2\varphi_n(x)
    &=      \sum\limits_{k=0}^n \left[ \frac{f_n(k)}{1-\frac{k}{n}x}
                + \frac{f_n(n-k)}{1 - \frac{n-k}{n} x} \right]
    \\&=    \sum\limits_{k=0}^n \left[ \frac{2 - x}
                {\left(1-\frac{k}{n}x\right)
                \left(1 - \frac{n-k}{n} x \right)} \right] f_n(k)
    \\&=    - \frac{1}{x-2}\sum\limits_{k=0}^n \frac{4 + \frac{x^2}{1-x} }
                {1 + \frac{k(n-k)}{n^2} \frac{x^2}{1-x}} f_n(k)
    \\&=    \frac{1}{x-2} \rho\left(\frac{x^2}{1-x}\right),
\end{align*}
which is an element of $\calF_1$ by Theorems \ref{thrm:Deg0HHS} and \ref{thrm:LaurentPoly}.

As the first consequence of this fact, by Theorem \ref{thrm:Relations},
we have that for each $m \geq 1$,
\begin{equation}
\label{eq:G-CRelat}
    \sum\limits_{i=0}^m (-1)^i {2m-1 \choose m-i}{m+i \choose i} S_{n,m+i-1} = 0.
\end{equation}
Following \cite[Equation (6)]{Gould}, we define
\[
    Q_i^m := {m\choose i} + 2{m\choose i-1} = \frac{m! (m+i+1)}{i!(m-i+1)!}
\]
and observe that by a simple computation,
\[
    {2m-1 \choose m-i}{m+i \choose i} = \frac{(2m-1)!}{m! (m-1)!} Q_i^{m-1} .
\]
Hence, after dividing out the constant $\frac{(2m-1)!}{m! (m-1)!}$,
Equation \eqref{eq:G-CRelat} can be rewritten as
\[
    \sum\limits_{i=0}^m (-1)^i Q_i^{m-1} S_{n,m+i-1} = 0.
\]
Substituting Equation \eqref{eq:G-CSnpDef} and multiplying by $n^{2m-1}$ yields
\[
    \sum\limits_{i=0}^m (-1)^i Q_i^{m-1} n^{m-i} \sum\limits_{k=0}^n k^{m+i-1} f_n(k) = 0.
\]
This is precisely \cite[Equation (5)]{Gould}, where in the reference the term corresponding
to $i = m$ is separated to one side of the equation.

As the second application of the fact that $\varphi_n(x) \in \calF_1$,
Theorem \ref{thrm:OddFromEven} implies
\begin{equation}
\label{eq:G-COddEven}
    S_{n,2m+1} = \sum\limits_{i=0}^m {m \brace i} S_{n, 2i}.
\end{equation}
We claim that up to a constant factor, this is equivalent to \cite[Equation (15)]{Gould},
which can be expressed using Equation \eqref{eq:G-CSnpDef} as
\begin{equation}
\label{eq:G-C15}
    2^{\lfloor (m+3)/2\rfloor} S_{n,2m+1} = \sum\limits_{i=0}^m A_i^m S_{n,2i},
\end{equation}
where the coefficients $A_i^m$ are defined recursively in \cite[Equation (16)]{Gould}:
\[
    2^{-\lfloor (k+2)/2\rfloor}A_i^{k+1}
    =
    {2k+3\choose 2i} - \sum\limits_{j=i}^k {2k+3\choose 2j+1}
    2^{-\lfloor (j+3)/2\rfloor} A_i^j,
    \quad\quad 0\leq i\leq k
\]
with seed values $A_i^i = (2i+1) 2^{\lfloor (i+1)/2\rfloor}$, $i\ge 0$.

To see this, using \cite[Equation (18)]{Gould} as well as \cite[Equation (1.5)]{Carlitz}
to express the $A_0^{m-i}$ in terms of Bernoulli numbers, we have that for integers $m$ and $i$
with $m\geq i \geq 0$,
\begin{align*}
    A_i^m
    &= \begin{cases}
        {2m+1\choose 2m-2i+1} 2^{\lfloor i/2\rfloor} A_0^{m-i},
            &   \mbox{$m-i$ odd},
            \\
        {2m+1\choose 2m-2i+1} 2^{\lfloor (i+1)/2\rfloor} A_0^{m-i},
            &   \mbox{$m-i$ even},
    \end{cases}
    \\&= \begin{cases}
        {2m+1\choose 2i} 2^{\lfloor i/2\rfloor}
            2^{(m-i+3)/2} \left(4^{m-i+1} - 1\right)\frac{B_{2(m-i+1)}}{m-i+1},
            &   \mbox{$m-i$ odd},
            \\
        {2m+1\choose 2i} 2^{\lfloor (i+1)/2\rfloor}
            2^{(m-i+2)/2} \left(4^{m-i+1} - 1\right)\frac{B_{2(m-i+1)}}{m-i+1},
            &   \mbox{$m-i$ even},
    \end{cases}
    \\&=
    {2m+1\choose 2i} 2^{\lfloor (m+3)/2\rfloor} \left(4^{m-i+1} - 1\right)\frac{B_{2(m-i+1)}}{m-i+1}
    \\&=
    2^{\lfloor (m+3)/2\rfloor}{m \brace i}.
\end{align*}
With this, Equations \eqref{eq:G-COddEven} and \eqref{eq:G-C15} are clearly equivalent.

As our final application of $\varphi_n(x) \in \calF_1$, we apply Theorem
\ref{thrm:EvenFromOdd} to obtain that for all $m \geq 0$,
\begin{equation}
\label{eq:G-C29our}
    S_{n,2m} = \frac{2}{2m+1}\sum\limits_{i=0}^m {2m+1\choose 2i+1} B_{2m-2i} S_{n,2i+1}.
\end{equation}
We set
\[
    C_m = 1\cdot 3\cdot 5 \cdot 7 \cdot \ldots \cdot (2m+1)
\]
following \cite[page 314]{Gould} and note that Gould's $G_i^m$ \cite[page 314]{Gould} are shown in
\cite[Equation (1.6)]{Carlitz} to be given by
\[
    G_i^m = 2\cdot 1\cdot 3 \cdot 5\cdot \ldots \cdot (2m-1) {2m+1\choose 2i+1} B_{2m-2i}.
\]
Then Equation \eqref{eq:G-C29our} is evidently equivalent to
\[
    C_m S_{n,2m} = \sum\limits_{i=0}^m G_i^m S_{n,2i+1},
\]
which recovers \cite[Equation (29)]{Gould}.


\section{Coefficient triangles}
\label{sec:CoefTriang}

In this section, we give a few examples of the first portions of the coefficient triangles
resulting from Equations \eqref{eq:DegreePositiveEven}, \eqref{eq:DegreePositiveOdd},
\eqref{eq:DegreeAnyFirst}, and \eqref{eq:DegreeNegativeOdd} to indicate
their structure. Of interest is the appearance of the \emph{Lucas triangle}, introduced in
\cite{Feinberg}, which satisfies the same recursion as Pascal's triangle but begins with
row $n=0$ with entry $2$ and row $n=1$ with entries $1, 2$, see Table \ref{tab:Lucas}.
For $n\geq 1$, the $j$th entry of the $n$th row of the Lucas triangle is given by
\begin{equation}
\label{eq:Lucas}
    \frac{n+j}{n}{n\choose j},
\end{equation}
see \cite[Equation (14)]{NevilleLucas}.

\begin{table}[h]
\small
\setlength{\tabcolsep}{-5pt}
\begin{tabular}{P{24pt}P{24pt}P{24pt}P{24pt}P{24pt}P{24pt}P{24pt}P{24pt}
    P{24pt}P{24pt}P{24pt}P{24pt}P{24pt}P{24pt}P{24pt}P{24pt}P{24pt}}
 &&&&&&&& $2$
 \\
 &&&&&&& $1$ && $2$
 \\
 &&&&&& $1$ && $3$ && $2$
 \\
 &&&&& $1$ && $4$ && $5$ && $2$
 \\
 &&&& $1$ && $5$ && $9$ && $7$ && $2$
 \\
 &&& $1$ && $6$ && $14$ && $16$ && $9$ && $2$
 \\
 && $1$ && $7$ && $20$ && $30$ && $25$ && $11$ && $2$
 \\
 & $1$ && $8$ && $27$ && $50$ && $55$ && $36$ && $13$ && $2$
 \\
 $1$ && $9$ && $35$ && $77$ && $105$ && $91$ && $49$ && $15$ && $2$
 \\
 &&&&&&&& $\;\;\cdots$
\end{tabular}
\caption{The Lucas triangle.}
\label{tab:Lucas}
\end{table}

For the coefficient triangles given below, the row labeled $i$
lists the coefficients of $\gamma_i, \gamma_{i+1}, \ldots$ in the corresponding constraint.

When $r = 0$, Equation \eqref{eq:DegreePositiveEven} evidently yields Pascal's triangle with
an alternating sign. When $r$ is even and positive, Equation \eqref{eq:DegreePositiveEven}
yields Pascal's triangle with an upper-left portion removed due to the fact that $\gamma_i = 0$
for $i < 0$. This is illustrated in Table \ref{tab:r8} with the case $r = 8$.

\begin{table}[h]
\small
\setlength{\tabcolsep}{-5pt}
\begin{tabular}{l P{24pt}P{24pt}P{24pt}P{24pt}P{24pt}P{24pt}P{24pt}P{24pt}P{24pt}
    P{24pt}P{24pt}P{24pt}P{24pt}P{24pt}P{24pt}P{24pt}P{24pt}P{24pt}P{24pt}P{24pt}}
 $0\!:$&&&&&&&&&&&&& $-4$ && $1$
 \\
 $0\!:$&&&&&&&&&& $10$ && $-10$ && $5$ && $-1$
 \\
 $0\!:$&&&&&&& $-6$ && $15$ && $-20$ && $15$ && $-6$ && $1$
 \\
 $0\!:$&&&& $1$ && $-7$ && $21$ && $-35$ && $35$ && $-21$ && $7$ && $-1$
 \\
 $1\!:$&&& $1$&& $-8$ && $28$ && $-56$ && $70$ && $-56$ && $28$ && $-8$ && $1$
 \\
 $2\!:$\hspace{2cm}&&
      $1$&& $-9$ && $36$ && $-84$ && $126$ && $-126$ && $84$ && $-36$ && $9$ && $-1$
 \\
 &&&&&&&&&&& $\;\;\cdots$
\end{tabular}
\caption{The coefficient triangle of Equation \eqref{eq:DegreePositiveEven} when $r=8$.}
\label{tab:r8}
\end{table}

When $r=1$, the beginning of the coefficient triangle is given in Table \ref{tab:r1}. Note that
many rows have a common factor, and dividing by the $\gcd$ of each row yields the Lucas triangle
with alternating signs. To see this, note that Equation \eqref{eq:DegreePositiveOdd} can be
rewritten as
\[
    \frac{(2m + r - 2)!}{m!(m-1)!}\sum\limits_{i=0}^{m}
        (-1)^i \frac{(m+i)!}{(m + i + r - 2)!m} {m \choose i} \gamma_{m + i - 1} = 0.
\]
Dividing out the constant $\frac{(2m + r - 2)!}{m!(m-1)!}$ and setting $r=1$ yields
Equation \eqref{eq:Lucas} with alternating signs.

\begin{table}[h]
\small
\setlength{\tabcolsep}{-3pt}
\begin{tabular}{lP{24pt}P{24pt}P{24pt}P{24pt}P{24pt}P{24pt}P{24pt}P{24pt}P{24pt}
    P{24pt}P{24pt}P{24pt}P{24pt}P{24pt}P{24pt}P{24pt}P{24pt}}
 $0\!:$&&&&&&&& $1$ &&  $-2$
 \\
 $1\!:$&&&&&&& $3$ && $-9$ && $6$
 \\
 $2\!:$&&&&&& $10$ && $-40$ && $50$ && $-20$
 \\
 $3\!:$&&&&& $35$ && $-175$ && $315$ && $-245$ && $70$
 \\
 $4\!:$&&&& $126$ && $-756$ && $1764$ && $-2016$ && $1134$ && $-252$
 \\
 $5\!:$&&& $462$ && $-3234$ && $\;\, 9240$ && $-13860$ && $\;\, 11550$ && $-5082$ && $924$
 \\
 $6\!:$&& $\;\; 1716$ && $-13728$ && $\;\, 46332$ && $-85800$ && $\;\, 94380$ && $-61776$ && $\;\, 22308$ && $-3432$
 \\
 &&&&&&&&& $\;\; \cdots$
\end{tabular}
\caption{The coefficient triangle of Equation \eqref{eq:DegreePositiveOdd} when $r=1$.}
\label{tab:r1}
\end{table}

When $r>1$ is odd and positive, a similar coefficient triangle is obtained, which is illustrated
with the case $r = 9$ in Table \ref{tab:r9}.

\begin{table}[h]
\small
\setlength{\tabcolsep}{-2pt}
\begin{tabular}{lP{24pt}P{24pt}P{24pt}P{24pt}P{24pt}P{24pt}P{24pt}P{24pt}P{24pt}
    P{24pt}P{24pt}P{24pt}P{24pt}P{24pt}P{24pt}P{24pt}P{24pt}}
 $0\!:$&&&&&&&& $9$ &&  $-2$
 \\
 $1\!:$&&&&&&& $55$ && $-33$ && $6$
 \\
 $2\!:$&&&&&& $286$ && $-312$ && $130$ && $-20$
 \\
 $3\!:$&&&&& $1365$ && $-2275$ && $1575$ && $-525$ && $70$
 \\
 $4\!:$&&&& $6188$ && $-14280$ && $14280$ && $-7616$ && $2142$ && $-252$
 \\
 $5\!:$&&& $27132$ && $-81396$ && $\;\; 108528$ && $-81396$ && $\;\; 35910$ && $-8778$ && $924$
 \\
 $6\!:$&& $116280$ && $-434112$ && $\;\; 732564$ && $-718200$ && $\;\; 438900$ && $-166320$ && $\;\; 36036$ && $-3432$
 \\
 &&&&&&&&& $\;\;\cdots$
\end{tabular}
\caption{The coefficient triangle of Equation \eqref{eq:DegreePositiveOdd} when $r=9$.}
\label{tab:r9}
\end{table}

When $r$ is even and negative, the result is eventually Pascal's triangle with alternating
signs but begins with a finite positive triangle corresponding to Equation
\eqref{eq:DegreeAnyFirst}; see Table \ref{tab:rm8}.

\begin{table}[h]
\setlength{\tabcolsep}{-3pt}
\begin{tabular}{lP{24pt}P{24pt}P{24pt}P{24pt}P{24pt}P{24pt}P{24pt}P{24pt}P{24pt}
    P{24pt}P{24pt}P{24pt}P{24pt}P{24pt}P{24pt}P{24pt}P{24pt}}
 $0\!:$&&&&& $\frac{1}{9}$ &&  $\frac{1}{36}$
 \\
 $1\!:$&&&& $\frac{1}{36}$ && $\frac{1}{42}$ && $\frac{1}{126}$
 \\
 $2\!:$&&& $\frac{1}{84}$ && $\frac{1}{42}$ && $\frac{1}{42}$ && $\frac{1}{84}$
 \\
 $3\!:$&& $\frac{1}{126}$ && $\frac{2}{63}$ && $\frac{1}{14}$ && $\frac{1}{9}$ && $\frac{1}{9}$
 \\ \\
 $9\!:$&&&&&& $1$
 \\
 $10\!:$&&&&& $1$ &&  $-1$
 \\
 $11\!:$&&&& $1$ && $-2$ && $1$
 \\
 $12\!:$&&& $1$ && $-3$ && $3$ && $-1$
 \\
 &&&&&&$\;\;\cdots$
\end{tabular}
\caption{The coefficient triangle of Equations \eqref{eq:DegreePositiveEven}
and \eqref{eq:DegreeAnyFirst} when $r=-8$.}
\label{tab:rm8}
\end{table}

When $r$ is odd and negative, Equation \eqref{eq:DegreeNegativeOdd} yields the Lucas triangle
with an alternating sign; this can easily be seen by comparing
Equation \eqref{eq:DegreeNegativeOdd} to Equation~\eqref{eq:Lucas}. Hence, the result starts with a finite
positive triangle given by Equation \eqref{eq:DegreeAnyFirst} followed by an alternating sign Lucas triangle;
see Table \ref{tab:rm9} for an example.

\begin{table}[h]
\setlength{\tabcolsep}{-3pt}
\begin{tabular}{lP{25pt}P{25pt}P{25pt}P{25pt}P{25pt}P{25pt}P{25pt}P{25pt}P{25pt}P{25pt}P{25pt}P{25pt}P{25pt}P{25pt}P{25pt}P{25pt}}
$0\!:$ &&&&&& $\frac{1}{10}$ &&  $\frac{1}{45}$
 \\
$1\!:$ &&&&& $\frac{1}{45}$ && $\frac{1}{60}$ && $\frac{1}{210}$
 \\
$2\!:$ &&&& $\frac{1}{120}$ && $\frac{1}{70}$ && $\frac{1}{84}$ && $\frac{1}{210}$
 \\
$3\!:$ &&& $\frac{1}{210}$ && $\frac{1}{63}$ && $\frac{1}{35}$ && $ \frac{1}{30}$ && $\frac{1}{45}$
 \\
$4\!:$ && $\frac{1}{252}$ && $\frac{1}{42}$ && $\frac{1}{12}$ && $\frac{2}{9}$ && $\frac{1}{2}$ && $1$
 \\ \\
$10\!:$\hspace{10pt} &&&&&& $1$ &&  $-2$
 \\
$11\!:$ &&&&& $1$ && $-3$ && $2$
 \\
$12\!:$ &&&& $1$ && $-4$ && $5$ && $-2$
 \\
&&&&&&&$\;\;\cdots$
\end{tabular}
\caption{The coefficient triangle of Equations \eqref{eq:DegreeAnyFirst}
and \eqref{eq:DegreeNegativeOdd} when $r=-9$.}
\label{tab:rm9}
\end{table}


\subsection{Hidden extra symmetries}
\label{subsec:HiddenSymm}

As a final note, we make a brief observation about an alternative rescaling of Equations
\eqref{eq:DegreeAnyFirst} and \eqref{eq:DegreePositiveEven} that yields an additional symmetry.

Suppose $r=1-2k$ is odd and negative. Then Equation \eqref{eq:DegreeAnyFirst} can be rewritten as
\begin{equation}
\label{eq:DegreeAnyFirstRescaled}
    \frac{m!(k-m)!}{(2k)!}\sum\limits_{i=0}^{m}
        \frac{(m+i)!(2k-m-i)!}{(m-i)!i!(k-m)!}\gamma_{m + i - 1} = 0,
\end{equation}
yielding a system of equations that is invariant under the substitution $(m,i) \mapsto (k-i,k-m)$.
That is, the corresponding finite positive triangle rescaled in this way has a slanted symmetry
where the lower left corner can be treated as the top of a symmetric number array. This is illustrated
in Table \ref{tab:rm9rescaled} which gives the rescaled upper triangle of Table \ref{tab:rm9}.

\begin{table}[h]
\setlength{\tabcolsep}{-3pt}
\begin{tabular}{lP{25pt}P{25pt}P{25pt}P{25pt}P{25pt}P{25pt}P{25pt}P{25pt}P{25pt}P{25pt}P{25pt}P{25pt}}
$0\!:$ &&&&&& $15120$ &&  $3360$
 \\
$1\!:$ &&&&& $6720$ && $5040$ && $1440$
 \\
$2\!:$ &&&& $2520$ && $4320$ && $3600$ && $1440$
 \\
$3\!:$ &&& $720$ && $2400$ && $4320$ && $5040$ && $3360$
 \\
$4\!:$ && $120$ && $720$ && $2520$ && $6720$ && $15120$ && $30240$
\\ \mbox{}
\end{tabular}
\caption{The rescaled upper triangle of Table \ref{tab:rm9}, corresponding to
Equation \eqref{eq:DegreeAnyFirstRescaled} when $r=-9$.}
\label{tab:rm9rescaled}
\end{table}

In the same way, for any $k \ge m$, Equation \eqref{eq:DegreePositiveEven}
can be rewritten as
\[
    \frac{(m-1)!(k-m)!}{(k-1)!}\sum\limits_{i=0}^{m-1} (-1)^i
        \frac{(k-1)!}{(m-i-1)!i!(k-m)!}\gamma_{m - r + i} = 0,
\]
yielding a system that is again invariant under the substitution $(m,i) \mapsto (k-i,k-m)$.
This demonstrates that multiplying each row by a suitable scalar transforms Pascal's
triangle into a fully symmetric array of multinomial coefficients; see Table~\ref{tab:PascalRescaled}.

\begin{table}[h]
\small
\setlength{\tabcolsep}{-5pt}
\begin{tabular}{P{24pt}P{24pt}P{24pt}P{24pt}P{24pt}P{24pt}P{24pt}P{24pt}
    P{24pt}P{24pt}P{24pt}P{24pt}P{24pt}P{24pt}P{24pt}P{24pt}P{24pt}}
 &&&&&&&& $1$
 \\
 &&&&&&& $8$ && $8$
 \\
 &&&&&& $28$ && $56$ && $28$
 \\
 &&&&& $56$ && $168$ && $168$ && $56$
 \\
 &&&& $70$ && $280$ && $420$ && $280$ && $70$
 \\
 &&& $56$ && $280$ && $560$ && $560$ && $280$ && $56$
 \\
 && $28$ && $168$ && $420$ && $560$ && $420$ && $168$ && $28$
 \\
 & $8$ && $56$ && $168$ && $280$ && $280$ && $168$ && $56$ && $8$
 \\
 $1$ && $8$ && $28$ && $56$ && $70$ && $56$ && $28$ && $8$ && $1$
 \\
 &&&&&&&& $\;\;\cdots$
\end{tabular}
\caption{Pascal's triangle rescaled with respect to row $n=8$.}
\label{tab:PascalRescaled}
\end{table}


\bibliographystyle{amsplain}
\bibliography{HHS-Gorenstein}

\providecommand{\bysame}{\leavevmode\hbox to3em{\hrulefill}\thinspace}
\providecommand{\MR}{\relax\ifhmode\unskip\space\fi MR }
\providecommand{\MRhref}[2]{%
  \href{http://www.ams.org/mathscinet-getitem?mr=#1}{#2}
}
\providecommand{\href}[2]{#2}
\begin{thebibliography}{10}

\bibitem{AbramowitzStegun}
Milton Abramowitz and Irene~A. Stegun, \emph{Handbook of mathematical functions
  with formulas, graphs, and mathematical tables}, National Bureau of Standards
  Applied Mathematics Series, vol.~55, For sale by the Superintendent of
  Documents, U.S. Government Printing Office, Washington, D.C., 1964.

\bibitem{BensonBook}
D.~J. Benson, \emph{Polynomial invariants of finite groups}, London
  Mathematical Society Lecture Note Series, vol. 190, Cambridge University
  Press, Cambridge, 1993.

\bibitem{BrunsHerzog}
Winfried Bruns and J{\"u}rgen Herzog, \emph{Cohen-{M}acaulay rings}, Cambridge
  Studies in Advanced Mathematics, vol.~39, Cambridge University Press,
  Cambridge, 1993.

\bibitem{Carlitz}
L.~Carlitz, \emph{Note on certain triangular arrays}, SIAM J. Math. Anal.
  \textbf{1} (1970), 328--332.

\bibitem{CHHS}
Emily Cowie, Hans-Christian Herbig, Daniel Herden, and Christopher Seaton,
  \emph{The {H}ilbert series and {$a$}-invariant of circle invariants},
  (2017), \texttt{arXiv:1707.03128 [math.RA]}.

\bibitem{DerskenKemperBook}
Harm Derksen and Gregor Kemper, \emph{Computational invariant theory},
  Invariant Theory and Algebraic Transformation Groups, I, Springer-Verlag,
  Berlin, 2002, Encyclopaedia of Mathematical Sciences, 130.

\bibitem{Feinberg}
Mark Feinberg, \emph{A {L}ucas triangle}, Fibonacci Quart. \textbf{5} (1967),
  486--490.

\bibitem{GesselUmbral}
Ira~M. Gessel, \emph{Applications of the classical umbral calculus}, Algebra
  Universalis \textbf{49} (2003), no.~4, 397--434, Dedicated to the memory of
  Gian-Carlo Rota.

\bibitem{Gould}
H.~W. Gould, \emph{Power sum identities for arbitrary symmetric arrays}, SIAM
  J. Appl. Math. \textbf{17} (1969), 307--316.

\bibitem{HHS}
Hans-Christian Herbig, Daniel Herden, and Christopher Seaton, \emph{On
  compositions with $x^2/(1-x)$}, Proc. Amer. Math. Soc. \textbf{143} (2015),
  4583--4596.

\bibitem{HerbigSeatonEM}
Hans-Christian Herbig and Christopher Seaton, \emph{The {H}ilbert series of a
  linear symplectic circle quotient}, Exp. Math. \textbf{23} (2014), no.~1,
  46--65.

\bibitem{Moll}
Victor~H. Moll, \emph{Numbers and functions}, Student Mathematical Library,
  vol.~65, American Mathematical Society, Providence, RI, 2012, From a
  classical-experimental mathematician's point of view.

\bibitem{Mordell}
L.~J. Mordell, \emph{The sign of the {B}ernoulli numbers}, Amer. Math. Monthly
  \textbf{80} (1973), 547--548.

\bibitem{PopovVinberg}
V.~L. Popov and {\`E}.~B. Vinberg, \emph{Invariant theory}, Algebraic geometry.
  {IV}, Encyclopaedia of Mathematical Sciences, vol.~55, Springer-Verlag,
  Berlin, 1994, Linear algebraic groups. Invariant theory, A translation of
  {{\i}t Algebraic geometry. 4} (Russian), Akad. Nauk SSSR Vsesoyuz. Inst.
  Nauchn. i Tekhn. Inform., Moscow, 1989, Translation edited by A. N. Parshin
  and I. R. Shafarevich, pp.~vi+284.

\bibitem{NevilleLucas}
Neville Robbins, \emph{The {L}ucas triangle revisited}, Fibonacci Quart.
  \textbf{43} (2005), no.~2, 142--148.

\bibitem{Stanley}
Richard~P. Stanley, \emph{Hilbert functions of graded algebras}, Advances in
  Math. \textbf{28} (1978), no.~1, 57--83.

\bibitem{SturmfelsBook}
Bernd Sturmfels, \emph{Algorithms in invariant theory}, Texts and Monographs in
  Symbolic Computation, Springer-Verlag, Vienna, 1993.

\bibitem{VillarrealMonomAlgs}
Rafael~H. Villarreal, \emph{Monomial algebras}, second ed., Monographs and
  Research Notes in Mathematics, CRC Press, Boca Raton, FL, 2015.

\bibitem{WatanabeGor1}
Keiichi Watanabe, \emph{Certain invariant subrings are {G}orenstein. {I}},
  Osaka J. Math. \textbf{11} (1974), 1--8.

\bibitem{WatanabeGor2}
\bysame, \emph{Certain invariant subrings are {G}orenstein. {II}}, Osaka J.
  Math. \textbf{11} (1974), 379--388.

\end{thebibliography}

\end{document}